\documentclass[11pt,intlimits,reqno]{amsart}

\usepackage{hyperref}
\usepackage{graphicx}
\usepackage{color}
\usepackage{amsfonts,amssymb}
\usepackage{amsthm}
\usepackage{bbm} 

\usepackage{mathrsfs}

\usepackage{a4wide}

\newtheorem{neu}{}[section]
\newtheorem{Cor}[neu]{Corollary}
\newtheorem*{Cor*}{Corollary}
\newtheorem{Thm}[neu]{Theorem}
\newtheorem*{Thm*}{Theorem}
\newtheorem{Prop}[neu]{Proposition}
\newtheorem*{Prop*}{Proposition}
\theoremstyle{definition}
\newtheorem{Lemma}[neu]{Lemma}
\newtheorem*{Rmk*}{Remark}
\newtheorem{Rmk}[neu]{Remark}
\newtheorem{Ex}[neu]{Example}
\newtheorem*{Ex*}{Example}

\newtheorem{Def}[neu]{Definition}

\newcommand{\N}{\mathbb{N}}

\newcommand{\Z}{\mathbb{Z}}
\newcommand{\R}{\mathbb{R}}
\newcommand{\C}{\mathbb{C}}
\newcommand{\CP}{\C\mathrm{P}}

\newcommand{\pf}{\longrightarrow}

\newcommand{\wrt}{with respect to }

\newcommand{\CZ}{\mu_{\mathrm{CZ}}}
\newcommand{\Mas}{\mu_{\mathrm{Maslov}}}
\newcommand{\Morse}{\mu_{\mathrm{Morse}}}

\newcommand{\id}{\mathrm{id}}

\newcommand{\om}{\omega}
\newcommand{\ev}{\mathrm{ev}}

\newcommand{\Poincare}{Poincar\'{e} }

\newcommand{\A}{\mathcal{A}}

\renewcommand{\P}{\mathcal{P}}

\newcommand{\F}{\mathcal{F}}

\newcommand{\D}{\mathbb{D}}

\newcommand{\M}{\mathcal{M}}
\newcommand{\Mh}{\widehat{\mathcal{M}}}

\newcommand{\J}{\mathcal{J}}
\newcommand{\B}{\mathcal{B}}
\newcommand{\E}{\mathcal{E}}
\renewcommand{\L}{\mathcal{L}}

\renewcommand{\H}{\mathrm{H}}

\newcommand{\CF}{\mathrm{CF}}

\newcommand{\HF}{\mathrm{HF}}

\newcommand{\Crit}{\mathrm{Crit}}
\newcommand{\CC}{\mathfrak{C}}

\newcommand{\beq}{\begin{equation}}
\newcommand{\beqn}{\begin{equation}\nonumber}
\newcommand{\eeq}{\end{equation}}

\newcommand{\bea}{\begin{equation}\begin{aligned}}
\newcommand{\bean}{\begin{equation}\begin{aligned}\nonumber}
\newcommand{\eea}{\end{aligned}\end{equation}}

\newcommand{\HH}{{\widehat{H}}}
\newcommand{\LH}{{L^N}}

\numberwithin{equation}{section}

\definecolor{Urs}{rgb}{0,.7,0}
\definecolor{Peter}{rgb}{0,0,1}
\definecolor{red}{rgb}{1,0,0}

\begin{document}
\title[Floer homology for negative line bundles]{Floer homology for negative line bundles and Reeb chords in pre-quantization spaces}
\author{Peter Albers} 
\author{Urs Frauenfelder}
\address{
    Peter Albers\\
    Departement Mathematik\\
    ETH Z\"urich}
\email{palbers@math.purdue.edu}
\address{
    Urs Frauenfelder\\
    Department of Mathematics and Research Institute of Mathematics\\
    Seoul National University}
\email{frauenf@snu.ac.kr}
\keywords{Floer homology, negative line bundles, Reeb chords, pre-quantization spaces}
\subjclass[2000]{53D40, 53D12, 37J99}
\begin{abstract}
In this article we prove existence of Reeb orbits for Bohr-Sommerfeld Legendrians in certain pre-quantization spaces. We give a quantitative estimate
from below.  These estimates are obtained by studying Floer homology for fibre-wise quadratic Hamiltonian functions on negative line bundles.
\end{abstract}
\maketitle

\section{Introduction}

\noindent In this article we consider a closed, connected symplectic manifold $(M,\om)$, which is integral, that is, $[\om]\in\H^2(M;\Z)$. Furthermore, let $L\subset M$ be a closed Lagrangian submanifold. Throughout this article assume that the pair $(M,L)$ is symplectically aspherical (see equations \eqref{eqn:def_sympl_aspherical} and \eqref{eqn:aspherical_Lagrangian} for the definition).

\begin{Def}\label{def:Bohr_Sommerfeld_intro}
A pair $(E,\alpha)$ consisting of a complex line bundle $E\pf M$ and a connection one form $\alpha$ is called a Bohr-Sommerfeld pair
for $(M,\om,L)$ if
\begin{enumerate}
  \item $\exists N\in\N$ s.t.~the curvature of $\alpha$ satisfies $F_{\alpha}=N\om$,
  \item the holonomy $\mathrm{hol}_{\alpha|_L}:\pi_1(L)\pf S^1$ takes values only in $\{0,\frac12\}\subset S^1=\R/\Z$.
\end{enumerate}
The integer $N=N(E,\alpha)$ is called the power of the Bohr-Sommerfeld pair.
\end{Def}

Pre-quantization spaces and Bohr-Sommerfeld pairs naturally arise in geometric quantization theory. Both notions appear in various places in the literature. For the Lagrangian case of Bohr-Sommerfeld we refer the reader for instance to Eliashberg-Hofer-Salamon \cite{Eliashberg_Hofer_Salamon}, Eliashberg-Polterovich \cite{Eliashberg_Polterovich_Partially_ordered_groups_and_geometry_of_contact_transformations}, and Ono \cite{Ono_Lagrangian_intersection_under_Legendrian_deformations}.

To a Bohr-Sommerfeld pair $(E,\alpha)$ for $(M,\om,L)$ we naturally associate a Legendrian submanifold $\L$ in a pre-quantization space of $(M,\om)$
as follows. The hyperplane distribution $\widetilde{\xi}:=\ker\alpha$ restricted to the unit circle bundle $\widetilde{\Sigma}$ of $E$ is
a contact structure on $\widetilde{\Sigma}$. Condition (2) in Definition \ref{def:Bohr_Sommerfeld_intro} implies that $L$ lifts to
a Legendrian submanifold $\widetilde{\L}$ of $(\widetilde{\Sigma},\widetilde{\xi})$. The group $\Z/2$ acts on
$(\widetilde{\Sigma},\widetilde{\xi},\widetilde{\L})$ by $e\mapsto -e$. The quotient is denoted by $(\Sigma,\xi,\L)$. We note that $\L$ is diffeomorphic to $L$. This is not the case if we don't divide out by the $\Z/2$-action.

Given a positive, autonomous Hamiltonian function $H\in C^{\infty}(M)$ on the base $M$ we denote by $\alpha_H$ the contact form on $(\Sigma,\xi)$
which is induced by the $S^1$-invariant contact form $\frac{1}{NH}\alpha$ on $\widetilde{\Sigma}$. We denote by $\mathcal{R}_\L(H)$ the set of
Reeb chords of the triple $(\Sigma,\alpha_H,\L)$ and by $\mathcal{R}^{1}_\L(H)$ the set of Reeb
chords of period strictly less than 1.\footnote{
The contact from $\alpha_H$ determines uniquely the Reeb vector field $R_H$ by $\alpha_H(R_H)=1$ and $\iota_{R_H}d\alpha_H=0$. Then a Reeb chord of period $T>0$ is a map $e:[0,T]\pf\Sigma$ solving $\dot{e}=R_H(e)$ and $e(0),e(T)\in\L$.
}
The set of contractible intersection points $L\cap\phi_H^1(L)$ of $L$ and its image under the time-1-map $\phi_H^1$ of the Hamiltonian flow of $H$
is denoted by $\P_L(H)$.\footnote{
The set of intersection points $L\cap\phi_H^1(L)$ is in 1-1-correspondence to the set of Hamiltonian chords $x(t)=\phi^t_H(x(0))$, $x(0),x(1)\in L$.
An intersection point is \textit{contractible} if the corresponding chord $x$ satisfies $[x]=0\in\pi_1(M,L)$.
}

The close connection between Reeb chords and Lagrangian intersection points was already fruitfully applied in the work of Eliashberg-Hofer-Salamon \cite{Eliashberg_Hofer_Salamon}, Givental \cite{Givental_Periodic_mappings_in_symplectic_topology, Givental_Nonlinear_generalization_of_the_Maslov_index,Givental_The_nonlinear_Maslov_index}, and Ono \cite{Ono_Lagrangian_intersection_under_Legendrian_deformations}. 

Our first main result gives a lower bound on the number of Reeb chords of period less than 1 in terms of the number of Hamiltonian chords of period equal to 1. The proof uses the observation that Reeb chords are in 1-1 correspondence to Hamiltonian chords with quantized action, see Proposition \ref{pro:quantized_chords_are_Reeb_chords}. Our result shows that in a certain sense the time-one dynamics ``remembers the past'' as phrased by Leonid Polterovich.

We recall that a subset of a topological space is called generic if it is a countable intersection of open and dense sets. It follows from Baire's theorem that generic subsets of $C^{\infty}(M)$ are dense. To a Hamiltonian function $H:M\pf\R$ we assign the following finite data set
\beq
\mathscr{D}(H):=\big\{\big(\A_H(x),\Mas^L(x;H)\big)\mid x\in\P_L(H)\big\}\;,
\eeq
where $\A_H$ is the action functional (see equation \eqref{eqn:def_action_functional}) and $\Mas^L$ is the Maslov index as defined in \cite{Robbin_Salamon_Maslov_index_for_paths}.

\noindent \textbf{Theorem A.}
Let $\dim M\geq4$. Then there exists a generic subset of  $C^{\infty}(M)$ such that for each Hamiltonian function $H$ in this subset there exist constants $C=C(\mathscr{D}(H))>0$ and $N=N(\mathscr{D}(H))\in\N$ with the following property. For any Bohr-Sommerfeld pair $(E,\alpha)$ with associated Legendrian $\L$ and power $N(E,\alpha)\geq N$ we have the estimate
\beq\label{eqn:crucial_inequality}
\#\mathcal{R}^{1}_\L(H+c)\geq\tfrac12\#\P_L(H)\
\eeq
for all $c\geq C$.

\begin{Rmk*}$ $
\begin{itemize}
\item In Section \ref{sec:applications} we introduce the two notions of a huge and a non-resonant Hamiltonian function. Moreover, we define the wiggliness $\mathcal{W}(\mathscr{D}(H))\in\N$ of a Hamiltonian function. Then in Theorem A we have $N(\mathscr{D}(H))=\mathcal{W}(\mathscr{D}(H))$ and $C(\mathscr{D}(H))$ is so that $H+C$ is huge. In fact, any Hamiltonian function $H$ becomes huge after adding a sufficiently large constant. Moreover, the wiggliness of a Hamiltonian function $H$ is large if $H$ has 1-periodic orbits with small but non-zero difference in action values. Finally, the non-resonancy condition is the generic property appearing in Theorem A. It guarantees that the action functionals detecting intersection points and Reeb chords are Morse.
\item We point out that Reeb dynamics of $\alpha_{H+c}$ (in particular the number $\#\mathcal{R_\L}(H+c)$) is sensitive to adding constants $c$ while $\P_L(H)$ is unaffected.
\item In fact, the period of the Reeb chords found in Theorem A is bounded below by a constant $\tau(H)>0$ depending on the wiggliness
and the local behavior of $H$ near $L$. Moreover, we get information on the action of the Reeb chords. We refer the reader to
Theorem \ref{thm:main_theorem} for the full statement.
\item We note that the Bohr-Sommerfeld property is stable under taking tensor powers. In particular, whenever there exists a Bohr-Sommerfeld pair
$(E,\alpha)$ for $(M,\om,L)$ then a suitable high tensor power of $(E,\alpha)$ will satisfy the assumption of Theorem A.
\item The same techniques used to prove Theorem A can be adapted to obtain an analogue of Theorem A for the number of closed Reeb orbits in terms of the number of contractible fixed points. In the periodic case multiple covers of a Reeb orbit contribute to the count. However, it should be possible to use the information on action, period, and index to get estimates for the number of geometrically distinct Reeb orbits. This will be treated in the future.
\end{itemize}
\end{Rmk*}

Floer's theorem gives a lower bound for $\P_L(H)$ in topological terms of $L$. Thus, we obtain

\begin{Cor} \label{cor:intro}
Under the assumptions of Theorem A
\beq
\#\mathcal{R}^1_\L(H+c)\geq\tfrac12\sum_{i=0}^{\dim L} b_i(L;\Z/2)
\eeq
where $b_i=\dim\H_i(L;\Z/2)$ are the Betti numbers.
\end{Cor}

\begin{Rmk}
We point out that the estimate \eqref{eqn:crucial_inequality} does not hold in general. In section \ref{section:counterex} we construct a large class of examples of Bohr-Sommerfeld pairs of power 1 for which $\mathcal{R}^1_\L(H)=\emptyset$. 
\end{Rmk}

\begin{Rmk}\label{rmk:function_mu_nu}
For fixed Legendrian $\mathcal{L}$ in a pre-quantization space and $H:M\pf(0,\infty)$ Theorem A can be rephrased in terms of the function
\beq
\mu\equiv\mu_{\mathcal{L},H}:(-\min H,\infty)\pf\N_0,\qquad \mu(c):=\#\mathcal{R}^{1}_\L(H+c)\;.
\eeq
Namely, if the power of the pre-quantization space is large enough we have 
\beq
\mu(c)\geq \tfrac12\#\P_L(H)
\eeq
for sufficiently large $c$. Moreover, the example from Section \ref{section:counterex} mentioned in the Remark above implies that there exists $\mathcal{L}$ and $H$ such that
\beq
\mu(c)=0\qquad\forall c\leq0\;,
\eeq
see Remark \ref{rmk:example_mu}. The function $\mu$ should not be confused with the function
\beq
\nu\equiv \nu_{\mathcal{L},H}:(0,\infty)\pf\N_{\geq0},\qquad \nu(c):=\#\mathcal{R}^{1}_\L(cH)\;
\eeq
which has the following properties. $\nu$ is monotone increasing, moreover 
\beq
\nu(c)=0 \qquad
\eeq
for all $c$ smaller than the smallest period of a Reeb chord of $\alpha_H$. We point out that the function $\mu$ in general won't satisfy $\displaystyle\lim_{c\to(-\min H)}\mu(c)=0$. Moreover, since there is no relation between Reeb chords of $\alpha_H$ and $\alpha_{H+c}$ it is unlikely that $\mu$ is monotone.
\end{Rmk}

The method of proof for Theorem A is to study Floer homology of fiber-wise quadratic Hamiltonian functions on $E$. In fact, for the construction of
Floer homology itself the Hamiltonian function can be chosen as usual, namely any time-dependent nondegenerate Hamiltonian function. We construct a version of Floer homology for periodic orbits $\HF_*^N(H)$ and for chords with Lagrangian boundary conditions $\HF_*^N(H;L)$.

Here are some details of the construction. Let $E\pf M$ be a complex line bundle
with first Chern class $c_1(E)=-[\om]$. Then $E$ and its tensor powers $E^N$ can be endowed with the structure of a symplectic manifold being
convex at infinity. For a generic Hamiltonian function $H:S^1\times M\pf\R$ we define a finite-dimensional, $\Z$-graded $\Z/2$-vector space $\HF^N_*(H)$
which is associated to a fiber-wise quadratic lift of the Hamiltonian function $H$ to the bundle $E^N$.
$\HF^N_*(H)$ is defined as the Floer homology of the action functional of classical mechanics for the lift of $H$. For a Bohr-Sommerfeld Lagrangian $L\subset M$ we construct a Lagrangian lift $L^N\subset E^N$. Then $\HF_*^N(H;L)$ is the Lagrangian Floer homology of $L^N$ and the fiber-wise quadratic lift of $H$.

The homology $\HF^N_*(H)$ and $\HF_*^N(H;L)$ depends on both $H$ and $N$. By choosing $N=N(H)$ large enough $\HF^N_*(H)$ detects all periodic orbits of $H$ and $\HF_*^N(H;L)$ detects all Hamiltonian chords of $H$.\\[.5ex]
\textbf{Theorem B.}
Given a generic $H$ there exists a positive integer $N=N(H)$ such that
\beq
\dim\HF^N(H)=\#\P(H)
\eeq
where $\P(H)$ is the set of contractible 1-periodic orbits of the Hamiltonian vector field of $H$ and
\beq
\dim\HF^N(H;L)=\#\P_L(H)
\eeq
where $\P_L(H)$ is the set of contractible 1-periodic chords of $H$.\\[1ex]
\noindent Theorem B is proved as Proposition \ref{prop:negative_bundles_detect_orbits} (periodic case) and Proposition \ref{prop:huge_computes_homology} (Lagrangian case). Theorem A follows from the Lagrangian version of Theorem B in the following way. If $H$ is positive and autonomous then there exists a
compact perturbation of the quadratic lift of $H$ such that the action functional of this
perturbation detects Reeb orbits resp.~chords. Theorem A follows then from Theorem B together with the invariance of Floer homology under compact perturbations. The factor $\tfrac12$ in Theorem A is due to the $\Z/2$-symmetry which was divided out to obtain the space $\Sigma$ from $\widetilde{\Sigma}.$

\begin{Rmk}
The periodic case of Theorem B could be used to prove a periodic version of Theorem A. Unfortunately, as such it's not very interesting because Reeb orbits can be iterated and iterates potentially contribute to the set $\mathcal{R}^{1}(H)$. However, since we have additional information about period and action of the Reeb orbits a refined analysis 
should lead also to non-trivial estimates in the periodic case. We plan to treat this in the future.
\end{Rmk}

\tableofcontents

\subsubsection*{Organization of the article}
In Section \ref{sec:general_FH} we review the construction of classical Floer homology.
In Section \ref{sec:FH_for_negative_lb_periodic_case} we construct Floer homology for negative line bundles in the periodic case.
In Subsection \ref{sec:the_setting} we describe the symplectic geometry of negative line bundles and introduce the notion of strongly nondegenerate
Hamiltonian function. The necessary $C^0$-estimates are proved in Subsection \ref{sec:convexity}. We show a subharmonic estimate
which generalizes the known results in symplectic homology. In Subsection \ref{sec:index_considerations} we compare the indices
of the action functional of classical mechanics on the base and on the total space of the bundle. We define the new Floer homology and corresponding
continuation homomorphism in Subsection \ref{sec:def_of_FH_for_neg_bdls}. Theorem B is proved as Proposition \ref{prop:negative_bundles_detect_orbits} and Proposition \ref{prop:huge_computes_homology}. In Section \ref{sec:HF^E_relative_case} we treat the construction of Floer homology of negative line bundles in the Lagrangian case. For this we
extend the previously proved $C^0$-estimates to Lagrangian boundary conditions by a reflection argument.
Section \ref{sec:applications} contains the applications to Hamiltonian/Reeb chords. Theorem A is a special case of Theorem \ref{thm:main_theorem}.
In Appendix \ref{appendix:non_resonant} we prove that being non-resonant is a generic property in dimensions higher than 2.
In Appendix \ref{appendix:autonomous_Lagranians} we prove a Poincar\'{e}-type theorem for the local behavior of Hamiltonian chords.
In Appendix \ref{appendix:quantized_chords} we prove a Morse condition for the perturbed action functional. Finally, in Appendix \ref{appendix:holonomy}
we collect some well-known facts about holonomy of tensor products of line bundles.

\subsubsection*{Acknowledgments}
This paper was written during visits of the first author to the Ludwig-Maximilians-Universit\"at M\"unchen and visits of the second author
to the Courant Institute, NYU. Both authors thank the institutions for their stimulating working atmospheres. We want to express our gratitude
to Kai Cieliebak and Helmut Hofer for helpful discussions, and to Viktor Ginzburg for enlightening remarks on an early version of our result. Moreover, we greatly profited from Leonid Polterovich's insightful comments. Finally, we want to thank  Paul Biran for pointing out example \ref{Ex:Biran_Ex} to us.

The authors are supported by the German Research Foundation (DFG) through Priority Programm 1154
"Global Differential Geometry", grants AL 904/1-1 and FR 2637/1-1. Moreover, they received support from NSF Grant DMS-0603957.

\section{Floer homology for closed symplectic manifolds}
\label{sec:general_FH}

\subsection{The periodic case}\label{sec:Floer_hom_closed_case}

In this section we briefly recall Floer's construction of his semi-infinite dimensional Morse homology on the free loop space. We follow closely Dietmar Salamon's lecture notes \cite{Salamon_lectures_on_floer_homology}.
Let $(M,\om)$ be a closed connected symplectic manifold. We assume for simplicity that $(M,\om)$ is symplectically aspherical, that is
\beq\label{eqn:def_sympl_aspherical}
c^{TM}_1|_{\pi_2(M)}=0\quad\text{and}\quad\om|_{\pi_2(M)}=0\,.
\eeq
For a time-dependent Hamiltonian function $H\in C^{\infty}(S^1\times M)$ we set $H_t:=H(t,\cdot)\in C^{\infty}(M)$ for $t\in S^1:=\R/\Z$.
The time-dependent vector field $X_{H_t}=X_H(t,\cdot)$ defined by
\beq
\om(X_{H_t},\cdot)=dH_t(\cdot)
\eeq
is called the Hamiltonian vector field of $H$. We denote by $\mathscr{L}$ the set of smooth, contractible 1-periodic loops in $M$.
The subset of contractible 1-periodic orbits of $X_H$ is denoted by
\beq
\P^1(H):=\big\{\,x\in \mathscr{L}\mid \dot{x}(t)=X_H\big(t,x(t)\big)\big\}\,.
\eeq
Elements $x\in\P^1(H)$ will also be referred to as (contractible) 1-periodic orbits of $H$. They are the critical points of the action functional
of classical mechanics $\A_H:\mathscr{L}\pf\R$ defined by
\beq\label{eqn:def_action_functional}
\A_H(x)=-\int_{\D^2}\bar{x}^*\om-\int_0^1H\big(t,x(t)\big)dt
\eeq
where $\bar{x}:\D^2\pf M$ is an extension of the contractible loop $x$ to the unit disk $\D^2$. Since $(M,\om)$ is symplectically aspherical
the definition of $\A_H$ does not depend on the choice of an extension. The Hamiltonian vector field $X_H$ defines a flow $\varphi_H^t$ of
symplectomorphisms of $(M,\om)$. The Hamiltonian function $H$ is called nondegenerate if
\beq\label{eqn:non_degenerate_hamiltonian}
\det \left(D\varphi_H^1(x(0))-\mathbbm{1}\right)\not=0
\eeq
for all $x\in\P^1(H)$. This is implied by the requirement that $\mathrm{graph}(\varphi_H^1)$ intersects the diagonal in $M\times M$ transversally.
However, the latter condition is stronger since it implies \eqref{eqn:non_degenerate_hamiltonian} for all periodic orbits rather than only for
contractible ones. Contractible periodic orbits of a nondegenerate Hamiltonian function are isolated. Thus, $\#\P^1(H)<\infty$ since $M$ is closed.
To each periodic orbit $x\in\P^1(H)$ the Conley-Zehnder index $\CZ(x;H)\in\Z$ is assigned. This is well-defined as an integer due to the symplectic asphericity of $(M,\om)$. The Conley-Zehnder index is normalized so that for a $C^2$-small Morse function $f$ we have
\beq
\CZ(x)=\Morse(x)-n \quad\forall x\in\Crit(f)\,.
\eeq

For a nondegenerate Hamiltonian function $H$ Floer's complex $(\CF_*(H),\partial(J,H))$ is defined as follows.
$\CF_k(H)$ is generated over the field $\Z/2$ by all periodic orbits with Conley-Zehnder index equal to $k$
\beq
\CF_k(H)=\bigoplus_{\substack{x\in\P^1(H)\\\CZ(x)=k}}\Z/2\,\left<x\right>\;.
\eeq
To define the differential $\partial(J,H)$ we choose an $S^1$-family of $\om$-compatible almost complex structures $J=J(t,\cdot)$ and consider
solutions to Floer's equation, that is, maps $u:\R\times S^1\pf M$ satisfying
\beq\left\{
\begin{aligned}
\;\;&\partial_su+J(t,u)\big(\partial_tu-X_H(t,u)\big)=0\\
&u(-\infty)=x_-,\,u(+\infty)=x_+\in\P^1(H)
\end{aligned}\right.
\eeq
The space of solutions $\M(x_-,x_+;J,H)$ is called a moduli space. The energy
\beq
E(u):=\int_{-\infty}^{+\infty}\int_0^1|\partial_su|^2dt\,ds
\eeq
of elements $u\in\M(x_-,x_+;J,H)$ can be computed in terms of the action functional $\A_H$
\beq\label{eqn:energy_action_identity}
E(u)=\A_H(x_-)-\A_H(x_+)\;.
\eeq
Floer's equation can be interpreted as (a replacement for the ill-defined) negative gradient flow of the action functional $\A_H$.
The moduli space $\M(x_-,x_+;J,H)$ carries an $\R$-action $\sigma\ast u(s,t):=u(s+\sigma,t)$ which is free if $x_-\not=x_+$. From
the energy identity \eqref{eqn:energy_action_identity} it follows that $\M(x_-,x_-;J,H)$ contains only one element namely the $s$-independent
solution $x_-$.

\begin{Thm}[Floer]
For a generic family of almost complex structures $J=J(t,\cdot)$ all moduli spaces are smooth manifolds and
\beq
\dim\M(x_-,x_+;J,H)=\CZ(x_+;H)-\CZ(x_-;H)\;.
\eeq
\end{Thm}

\begin{Rmk}\label{rmk:compactness}
Since we do Morse theory for the action functional $\A_H$ on the loop space of a compact manifold, gradient trajectories can escape to infinity
in the loop space only if derivatives explode. By Floer's equation the only way this can happen is by bubbling-off of holomorphic spheres. Since
the symplectic manifold in question is symplectically aspherical, there are no non-constant holomorphic spheres, hence the necessary compactness
is achieved.
\end{Rmk}

\begin{Thm}[\cite{Floer_Morse_theory_for_Lagrangian_intersections}]
For $x,z\in\P^1(H)$ the moduli space
\beq
\Mh(x,z;J,H):=\M(x,z;J,H)/\R
\eeq
is compact if $\CZ(z;H)-\CZ(x;H)=1$ and compact up to simple breaking if $\CZ(z;H)-\CZ(x;H)=2$.  That is, in the latter case
it admits a compactification (denoted by the same symbol) such that the boundary decomposes as follows
\beq
\partial\Mh(x,z;J,H)=\bigcup_{\substack{y\in\P^1(H)}}\Mh(x,y;J,H)\times\Mh(y,z;J,H)\,.
\eeq
\end{Thm}
Counting elements of zero dimensional moduli space defines the differential $\partial=\partial(J,H)$
\bea
\partial x_-:=\sum_{\substack{y\in\P^1(H)\\\CZ(x_+)=\CZ(x_-)+1}}\#_2\Mh(x_-,x_+;J,H)\cdot x_+\;.
\eea

The previous theorems imply that the boundary operator $\partial$ is well-defined and satisfies $\partial^2=0$. This defines
Hamiltonian Floer homology of $H$
\beq
\HF_*(H):=\H_*(\CF_*(H),\partial(J,H))\;.
\eeq
As suggested by the notation, $\HF_*(H)$ does not depend on the chosen almost complex structure $J$. Furthermore, for Hamiltonian functions $H,K,L:S^1\times M\pf\R$ there exist canonical, grading preserving isomorphisms
\beq
m(K,H):\HF_*(H)\stackrel{\cong}{\pf}\HF_*(K)
\eeq
satisfying
\beq
m(L,K)\circ m(K,H)=m(L,H)\;.
\eeq
Hence, Floer homology does not depend (up to canonical isomorphisms) on the Hamiltonian function. Using the fact
that for a $C^2$-small Morse function Floer trajectories are in 1-1 correspondence to Morse trajectories the following theorem can be shown.
\begin{Thm}[\cite{Floer_Morse_theory_for_Lagrangian_intersections}]
\beq
\HF_*(H)\cong\H_{n-*}(M;\Z/2)\;.
\eeq
\end{Thm}

The maps $m(H_1,H_0)$ are called continuation homomorphisms and are constructed as follows. We choose a smooth 1-parameter family $H_s(t,x)$
of Hamiltonian functions such that $H_s=H_0$ for $s\leq0$ and $H_s=H_1$ for $s\geq1$. The set of solutions of
\beq\label{eqn:def_cont_morphism}\left\{
\begin{aligned}
\;\;&\partial_su+J(t,u)\big(\partial_tu-X_{H_s}(t,u)\big)=0\\
&u(-\infty)=x_-\in\P(H_0),\,u(+\infty)=x_+\in\P(H_1)
\end{aligned}\right.
\eeq
is denoted by $\M(x_-,x_+;J,H_s)$. Counting the elements of zero-dimensional components of the moduli spaces $\M(x_-,x_+;J,H_s)$
defines the map $m(H_1,H_0):\CF(H_0)\pf\CF(H_1)$ which is a chain map: $\partial(J,H_1)\circ m(H_1,H_0)=m(H_1,H_0)\circ\partial(J,H_0)$.
The induced map on homology is denoted by the same symbol. It can be shown
that the homomorphisms $m(H_1,H_0)$ on homology do not depend on the chosen 1-parameter family $H_s$. Moreover, an explicit inverse is given
by the map $m(H_0,H_1)$. We recall the following well-known energy identity.

\begin{Lemma}\label{lemma:energy_estimate_for_continuation}
For $u\in\M(x_-,x_+;J,H_s)$ holds
\beq
\A_{H_0}(x_-)-\A_{H_1}(x_+)=\int_{-\infty}^\infty\int_0^1\frac{\partial H_s}{\partial s}(u)dtds+\int_{-\infty}^\infty\int_0^1|\partial_s u|^2dtds
\eeq
\end{Lemma}

\begin{proof}
We refer to \cite{Schwarz_Matthias_Morse_homology} or \cite{Salamon_lectures_on_floer_homology}.
\end{proof}

\subsection{The relative case}\label{sec:Floer_hom_relative_case}

Historically, the relative case of Floer homology was treated in fact before the absolute case in Floer's seminal article \cite{Floer_Morse_theory_for_Lagrangian_intersections}.

As before $(M,\om)$ is a closed connected symplectic manifold. Let $L\subset M$ be a closed connected Lagrangian submanifold which is symplectically
aspherical, that is
\beq
\Mas|_{\pi_2(M,L)}=0\quad\text{and}\quad\om|_{\pi_2(M,L)}=0\,.
\eeq
We denote by $I$ the interval $[0,1]$ and let $H:I\times M\pf\R$ be a smooth Hamiltonian function. In this case the action functional $\A_H$ is defined
on the space of contractible paths
\beq
\mathscr{P}:=\big\{x\in C^\infty(I,M)\mid x(0),x(1)\in L\,; [x]=0\in\pi_1(M,L)\big\}\;.
\eeq
We denote $\D^2_+:=\{z\in\D^2\mid\mathrm{Im}(z)\geq0\}$. Then for each $x\in\mathscr{P}$ we can choose a map $\bar{x}:\D^2_+\pf M$ satisfying
$\bar{x}(e^{\pi it})=x(t)$ and $\bar{x}(\D^2_+\cap\R)\subset L$. As in the periodic case the action functional of classical mechanics
$\A_H:\mathscr{P}\pf\R$ is defined by
\beq
\A_H(x):=-\int_{\D^2_+}\bar{x}^*\om-\int_0^1H\big(t,x(t)\big)dt\;.
\eeq
This definition is independent of the choice of $\bar{x}$ since $L$ is symplectically aspherical. The set $\P_L^1(H)$ of critical points of $\A_H$
are called Hamiltonian chords, i.e.
\beq
\P_L^1(H)=\{x\in\mathscr{P}\mid\dot{x}(t)=X_H(t,x(t))\}\;.
\eeq
There is an injective map from $\P_L^1(H)$ into the set of intersection points $L\cap\varphi^1_H(L)$ given by the evaluation map $x\mapsto x(1)$.
By symplectic asphericity the Maslov index $\Mas$ induces a well-defined map
\beq
\begin{cases}
\mathscr{P}\pf\Z&\text{if }\dim L=\text{even}\\
\mathscr{P}\pf\frac12+\Z&\text{if }\dim L=\text{odd}
\end{cases}
\eeq
which we denote by $x\mapsto\Mas(x;H)$. Here, we use the Maslov index $\Mas$ with the following normalization. For $C^2$-small functions $f$ whose restriction to $L$ is Morse there is a 1:1 correspondence between the critical points $\Crit(f|_L)$ and Hamiltonian chords $\P_L^1(f)$. Then the Maslov index is normalized by 
\beq
\Mas=\Morse-\frac{n}{2}
\eeq
on corresponding Hamiltonian chords and critical points. We call the Hamiltonian function $H$ nondegenerate if
\beq
D\varphi_H^1(T_{x(0)}L)\pitchfork T_{x(1)}L
\eeq
holds for all $x\in\P_L^1(H)$. For nondegenerate $H$ the action functional $\A_H$ is Morse.
In this case we define Floer's complex $(\CF_*(H;L),\partial(J,H))$ as follows.
The set $\CF_k(H;L)$ is generated over the field $\Z/2$ by all Hamiltonian chords with Maslov index $k$
\beq
\CF_k(H;L)=\bigoplus_{\substack{x\in\P_L^1(H)\\\Mas(x;H)=k}}\Z/2\,\left<x\right>\;
\eeq
where $k\in\Z$ or $k\in\frac12+\Z$ according to $\dim L=\mathrm{even}$ or $\dim L=\mathrm{odd}$.

To define the differential we consider the moduli space $\M_L(x_-,x_+;J,H)$ of perturbed holomorphic strips, that is,
the set of solutions $u:\R\times[0,1]\pf M$ of Floer's equation with Lagrangian boundary conditions
\beq\left\{
\begin{aligned}
\;\;&\partial_su+J(t,u)\big(\partial_tu-X_H(t,u)\big)=0\\
&u(s,0),\,u(s,1)\in L\\
&u(-\infty)=x_-,\,u(+\infty)=x_+\in\P_L^1(H)
\end{aligned}\right.
\eeq
As in the periodic case blowing-up of derivatives in the interior leads to bubbling-off of holomorphic spheres. In addition, blowing-up of
derivatives might occur at the boundary of the strip. This gives rise to bubbling-off of homomorphic disks with boundary on the Lagrangian
submanifold $L$. Both of these phenomena are excluded by symplectic asphericity. In particular, the construction of Hamiltonian Floer homology
carries over unchanged to the Lagrangian case. This leads to the definition of Lagrangian Floer homology $\HF_*(H;L)$. Again using
continuation homomorphisms $m(K,H)$ it can be shown that Lagrangian Floer homology is independent of the Hamiltonian function. Floer proved
\begin{Thm}[Floer]\label{thm_floer}
\beq
\HF_*(H;L)\cong\H_{\frac{n}{2}-*}(L;\Z/2)\;.
\eeq
\end{Thm}

\section{Floer homology for negative line bundles - the periodic case}
\label{sec:FH_for_negative_lb_periodic_case}

\subsection{Negative line bundles}\label{sec:the_setting}

As in Section \ref{sec:Floer_hom_closed_case} we assume that
the symplectic manifold $(M,\om)$ is closed, connected and symplectically aspherical. Moreover, we require the symplectic form to be integral, i.e.
\beq
[\om]\in\H^2(M;\Z)\;.
\eeq
Therefore, for each $N\in\N$ we can choose a complex line bundle $E^N\stackrel{p}{\pf}M$ with first Chern
class $c_1(E^N)=-N[\om]$. 

We continue to use the convention $S^1=\R/\Z$. In particular, the Lie algebra equals $\R$. With this convention the action of $S^1$ on the bundle $E^N$ is given by 
\bea
S^1\times E^N&\pf E^N\\
(t,u)&\mapsto e^{2\pi it}u
\eea
On $E^N$ we define a symplectic form $\Omega$ as follows.
We choose a Hermitian connection 1-form $\alpha$ on $E^N\setminus M$ whose curvature $F_{\alpha}=d\alpha$ satisfies
\beq
F_{\alpha}=N\om\;.
\eeq
Furthermore, we fix the function $f(r)=\pi r^2+\tfrac1N$. Abbreviating $r=||e||$ the following 2-form
\beq
\Omega:=f'(r)\,dr\wedge\alpha+f(r)N\,p^*\om
\eeq
is a symplectic form on $E^N$. We note that this is well-defined  and satisfies $\Omega|_M=\om$ since $f'(0)=0$.
Furthermore, on $E^N\setminus M$ the symplectic form
can be written as $\Omega=d\big(f(r)\alpha\big)$. The vector field defined on $E^N\setminus M$
\beq\label{eqn:Liouville_vfield}
{X}:=\frac{f(r)}{f'(r)}\,\frac{\partial}{\partial r}
\eeq
is a \textit{Liouville vector field} for $\Omega$, that is $\L_{X}\Omega=\Omega$, or equivalently $f(r)\alpha=\iota_X\Omega$.
Here $\L$ denotes the Lie derivative. In particular, for all $c>\tfrac1N$
the manifold
\beq\label{eqn:def_of_Sigma_c}
\Sigma_c:=\{f(r)=c\}
\eeq
is of contact type. If we consider the canonical variable $\rho=\ln f(r)$ the Liouville vector field can be written as
\beq
{X}=\frac{\partial}{\partial \rho}\;.
\eeq
We note that the positive part of the symplectization of $\Sigma_c$ embeds into $E^N$ whereas the negative part only embeds partially.
For a nondegenerate Hamiltonian function $H:S^1\times M\pf\R$ we set
\beq\label{eqn:def_H_hat}
\widehat{H}(t,e)=N\cdot f(r)\cdot H\big(t,p(e)\big):S^1\times E^N\pf\R\,.
\eeq
The connection 1-form $\alpha$ induces a natural splitting of $TE^N$ into horizontal and vertical subspaces
\beq
T_eE^N=T_e^hE^N\oplus T_e^vE^N\;.
\eeq
Moreover, the projection $p$ gives rise to an isomorphism $T_e^hE^N\cong T_{p(e)}M$.
The horizontal component $X_\HH^h$ and the vertical component $X_\HH^v$
of the Hamiltonian vector field of $\HH$ \wrt $\Omega$ compute to
\begin{subequations}
\begin{align}\label{eqn:Ham_vfield_of_HH}
p_*X_\HH^h(t,e)&=X_H\big(t,p(e)\big)  \\[1ex]\label{eqn:Ham_vfield_of_HH_vertical}
X_\HH^v(t,e)&=-N\cdot H\big(t,p(e)\big)\cdot R(e)
\end{align}
\end{subequations}
where $R$ is the unique vertical vector field satisfying $\alpha(R)=1$. We note that $R$ restricts to the Reeb vector field
of the contact manifold $\Sigma_c$. Moreover, the projection of a 1-periodic solution of $X_\HH$ is a 1-periodic solution of $X_H$.

\begin{Rmk}
For notational convenience we do not record the integer $N$ in the notation of the function $f$, the symplectic form $\Omega$, the lift $\HH$, etc..
Moreover, the above construction is canonical in the sense that
\beq
E^N\otimes E^M=E^{N+M}\;,
\eeq
see Appendix \ref{appendix:holonomy}.
\end{Rmk}

\begin{Lemma}\label{lemma:lifted_flow_preserves_horizontal_distribution}
The flow $\phi_\HH^\tau$ preserves the Liouville vector field $X$, and thus the 1-form $\alpha$. This implies that the linearized flow is of the form
\beq
D\phi_\HH^\tau(e)=
\begin{pmatrix}
D\phi_H^\tau(e) & 0\\[0.5ex]
0 & \mathbbm{1}
\end{pmatrix}
\eeq
with respect to the splitting $TE^N\cong T^hE^N\oplus T^vE^N$. In particular, $D\phi_\HH^\tau$ maps horizontal vectors on horizontal vectors.
\end{Lemma}

\begin{proof}
Since $\alpha$ is an Hermitean connection form $dr$ vanishes on horizontal lifts, where $r$ denotes the radial coordinate.
Equations \eqref{eqn:Ham_vfield_of_HH} and $\eqref{eqn:Ham_vfield_of_HH_vertical}$ imply that $\phi^\tau_\HH$ preserves the radial coordinate $r$.
More precisely, we have the equality
\beq
\phi^\tau_\HH(ae)=a\phi^\tau_\HH(e)
\eeq
where $a\in\R_{>0}$ acts by multiplication in the fiber. This immediately implies that $D\phi_\HH^\tau$ preserves the vector
field $\frac{\partial}{\partial r}$ and thus $X$ according to equation \eqref{eqn:Liouville_vfield}. Thus, we conclude
\beq
f(r)\alpha(\xi)=\Omega(X,\xi)=\Omega(D\phi_\HH^\tau(X),D\phi_\HH^\tau(\xi))=\Omega(X,D\phi_\HH^\tau(\xi))=f(r)\alpha(D\phi_\HH^\tau(\xi))
\eeq
that is $(\phi_\HH^\tau)^*\alpha=\alpha$. Moreover, since $\phi_\HH^\tau$ preserves the radial coordinate and $dr$ vanishes on horizontal lifts
we know that
\beq
D\phi_\HH^\tau(T_e^hE^N\oplus<R>)=T_{\phi_\HH^\tau(e)}^hE^N\oplus<R>
\eeq
Then $(\phi_\HH^\tau)^*\alpha=\alpha$ immediately implies
\beq
D\phi_\HH^\tau(T_e^hE^N)=T_{\phi_\HH^\tau(e)}^hE^N.
\eeq
\end{proof}

\begin{Rmk}
Since the $r$-coordinate is preserved by the flow $\phi_\HH^\tau$ orbits are either entirely contained in the zero-section $M$ or do not intersect $M$ at all.
\end{Rmk}

The principal $S^1$-bundle $p:\widetilde{\Sigma}:=\{e\in E^N\mid||e||=1\}\pf M$ associated to $(E^n,\alpha)$ gives rise to a contact manifold $(\widetilde{\Sigma},\alpha)$. 
By definition $(\widetilde{\Sigma},\alpha)$ admits a canonical $S^1$-action. Any $S^1$-invariant contact form on $\widetilde{\Sigma}$ with the same co-orientation is of the form $\alpha_H=\frac{1}{NH}\alpha$ for some autonomous, positive and $S^1$-invariant function $H:\widetilde{\Sigma}\pf(0,\infty)$ which we
identify with a function $H:M=\widetilde{\Sigma}/S^1\pf(0,\infty)$.

We recall that for a Hamiltonian function $H$ on the base $M$ we define in equation\eqref{eqn:def_H_hat} the fiber-wise quadratic lift $\HH$ to $E^N$.
The following lemma establishes a relationship between the Reeb vector field $R_H$ of $\alpha_H$ and the Hamiltonian vector field of $\HH$.

\begin{Lemma}\label{lemma:Hamiltonian_vfield_equals_Reeb_vfield}
The Reeb vector field $R_H$ of $(\widetilde{\Sigma},\alpha_H)$ equals $-X_\HH$.
\end{Lemma}

\begin{proof}
If follows from equations \eqref{eqn:Ham_vfield_of_HH} and \eqref{eqn:Ham_vfield_of_HH_vertical} that
\beq
\alpha_H(-X_\HH)=1\,.
\eeq
Moreover, if we write $X_\HH=X_H-NHR$ according to the splitting $TE^N\cong TM\oplus T^vE^N$ the following holds.%
\bea
d\alpha_H(X_\HH,\cdot)&= \frac{1}{NH}d\alpha(X_H-NHR,\cdot)-\frac{1}{NH^2}dH\wedge \alpha(X_H-NHR,\cdot)\\
    &= \frac{1}{NH} d\alpha(X_H,\cdot)-\frac{1}{NH^2}dH(X_H-NHR)\alpha(\cdot)+\frac{1}{NH^2}\alpha(X_H-NHR)dH(\cdot)\\
    &= \frac{1}{NH}N\om(X_H,\cdot)-0-\frac{1}{H} dH(\cdot)\\
    &=0
\eea
We used $d\alpha=Np^*\om$, $d\alpha(R,\cdot)=0$, $\alpha(R)=1$, and $\alpha(X_H)=0=dH(X_H)=dH(R)$.
\end{proof}

\begin{Lemma}\label{lemma:basic_lemma}
We fix a bundle $p:E^N\pf M$.
\begin{enumerate}
\item Assuming that $H$ is nondegenerate, the following are equivalent.
\begin{enumerate}
\item $\HH$ is nondegenerate.
\item $\displaystyle\A_\HH(e)\not\in \frac1N\Z\quad\forall e\in\P^1(\HH)$.
\item All periodic orbits of $\HH$ are contained in the zero-section $M$ (and then are necessarily periodic orbits of $H$).
\end{enumerate}
\item
Moreover, if there exists a 1-periodic solution $e$ of $X_\HH$ which is \textit{not} contained in the zero-section $M$ then all orbits $z\cdot e$
obtained by fiber-wise multiplication by $z\in\C$ are 1-periodic solutions of $X_\HH$. In particular,
\beq
\A_\HH(e)=\A_H\big(p(e)\big)
\eeq
in both, the degenerate and the nondegenerate case.
\end{enumerate}
\end{Lemma}

\begin{proof}
Let $e(t)$ be a 1-periodic solution of $X_\HH$ and set $x(t)=p\big(e(t)\big)$. From equation \eqref{eqn:Ham_vfield_of_HH} it is apparent
that $x\in\P^1(H)$. We denote by $P_x^t:E^N_{x(0)}\pf E^N_{x(t)}$ parallel transport \wrt $\alpha$ along the path $x$ and by $P^{-t}_x$ its inverse.

Let us assume that $e(0)$ lies not in the zero-section. Since $e$ is 1-periodic we conclude that the angle $\angle\big(e(0),e(1)\big)\in\Z$
(due to our convention $S^1=\R/\Z$).

We will compute this angle in two steps. We consider $e_0(t):=P^{-t}_x\big(e(t)\big)\in E^N_{x(0)}$. Then the angle between
$e_0(1)$ and $P^1_x(e_0(1))=e(1)$ is given by the holonomy which equals, see equation \eqref{eqn:def_of_holonomy},
\beq\label{eqn:computation_angle}
-\mathrm{hol}_\alpha(\gamma)=\int_{S^1}e^*\alpha=\int_{\D^2}\bar{e}^*d\alpha=\int_{\D^2}\bar{e}^*p^*(N\om)=N\int_{\D^2}\bar{x}^*\om\;,
\eeq
where we choose $\bar{e}:\D^2\pf M$ such that $\bar{e}(\exp(2\pi it))=e(t)$. Because of equation $\eqref{eqn:Ham_vfield_of_HH_vertical}$ the path $e_0(t)$ in fiber $E^N_{x(0)}$ satisfies
\bea\label{eqn:proof_of_basic_lemma}
\dot{e}_0(t)&=-N\cdot H\big(t,x(t)\big)\cdot R(e_0(t))\\
    &=-N\cdot H\big(t,x(t)\big)\cdot i\cdot e_0(t)
\eea
where $i\cdot e_0(t)$ is multiplication by $i\in\C$ in the fiber $E^N_{x(0)}$. Thus, the angle $\angle\big(e_0(0),e_0(1)\big)$
between $e_0(0)=e(0)$ and $e_0(1)$
equals
\beq
N\int_0^1H\big(t,x(t)\big)dt\;.
\eeq
Then $\angle\big(e(0),e(1)\big)\in\Z$ is equivalent to
\beq
\A_H(x)\in \frac1N\Z
\eeq
Before we prove the Lemma we observe that given $x\in\P^1(H)$ and $e_0\in E_{x(0)}$ the following path
\beq\label{eqn:lifted_periodic_orbit}
e(t):=\exp\Big(2\pi i N\int_0^tH\big(\tau,x(\tau)\big)d\tau\Big)P_x^te_0
\eeq
solves the ODE
\beq
\dot{e}(t)=X_\HH(t,e(t))\;.
\eeq
Moreover, $e(1)=e(0)$ if and only if $\A_H(x)\in\frac1N\Z$ by the computation above.

We now prove part (2) of the Lemma.\\[.5ex]
From equation \eqref{eqn:proof_of_basic_lemma} it is apparent that if $e$ is a 1-periodic solution of $X_\HH$ then so is $z\cdot e$ for any
$z\in\C$. In particular, if $e$ does not lie in the zero section, by multiplication with $z\in\C$ we can fill the entire fibres over $p(e)$
with periodic orbits. Since the action functional is constant on this critical manifold part (2) follows.\\[1ex]
Let us prove part (1).\\[.5ex]
\underline{(b) implies (c)}: We show that not (c) implies not (b).
If there exists a 1-periodic orbit $e\in\P^1(\HH)$ not lying in the zero-section then the above discussion shows that $\A_H(p(e))\in \frac1N\Z$.
This implies not (b).\\[1ex]
\underline{(a) implies (b)}: We show that not (b) implies not (a).
Assume that there exists $e\in\P^1(\HH)$ with $\A_\HH(e)\in \frac1N\Z$. As we concluded above this implies that the fibers over $p(e)$ are filled
entirely by 1-periodic orbits. This clearly shows that $\HH$ is degenerate.\\[1ex]
\underline{(b) implies (a)}: The linearization of the time-1-map $\phi_\HH$ of the Hamiltonian $\HH$ at a fixed point $x$ in the zero-section $M$
is represented by the following matrix using the canonical splitting $T_xE^N=T_xM\oplus E^N_x$
\beq\label{eqn:matrix_without_star}
D\phi_\HH(x)=
\begin{pmatrix}
D\phi_H(x) & 0\\[0.5ex]
0 & e^{2\pi i\beta}
\end{pmatrix}
\eeq
where the angle $\beta=N\cdot\A_\HH(x)$ by the considerations from above. Since $H$ is assumed to be nondegenerate $D\phi_H(x)$
has no eigenvalue equal to $1$. Hence, $D\phi_\HH(x)$ has an eigenvalue equal to $1$ if and only if $\beta=N\cdot\A_\HH(x)\in\Z$.\\[1ex]
\underline{(b) implies (c)}: We assume not (b) and (c). In particular, there exists $x\in\P^1(\HH)$ which is entirely contained in $M$ and satisfies $\A_H(x)\in\frac1N\Z$. As we observed above the latter implies that $x$ can be lifted via  \eqref{eqn:lifted_periodic_orbit}  to a loop $e\in\P^1(\HH)$ for any $e_0\in E_{x(0)}$. This clearly contradicts (c) and concludes the proof of the Lemma.
\end{proof}

\begin{Rmk}\label{rmk:basic_rmk}
Let $H:M\pf\R$ be autonomous and $g:\R\pf\R$ a smooth function. We consider a 1-periodic orbit $e$ of $\HH_g:=g(\HH)$, that is $e$ solves
\beq
\dot{e}(t)=X_{\HH_g}\big(e(t)\big)=g'\big(\HH(e)\big)\,X_\HH\big(e(t)\big)\,.
\eeq
Since $\HH$ is autonomous we compute
\bea
\frac12\frac{d}{dt}\Big(g'\big(\HH(e)\big)\Big)^2&=g'\big(\HH(e)\big)\cdot g''\big(\HH(e)\big)\cdot d\HH_{e}(\dot{e})\\
&=g'\big(\HH(e)\big)\cdot g''\big(\HH(e)\big)\cdot \om\Big(X_\HH(e),\dot{e}\Big)\\
&=g''\big(\HH(e)\big)\cdot \om\Big(g'\big(\HH(e)\big)\cdot X_\HH(e),\dot{e}\Big)\\
&=g''\big(\HH(e)\big)\cdot \om\big(\dot{e},\dot{e}\big)=0\\
\eea
In particular, $g'\big(\HH(e)\big)$ is constant. This implies that the projection $x(t)=p(e(t))$ is
(after reparametrization) a periodic orbit of $H$ with period $g'(\HH(e))$. In case that $e$ is not contained in the zero section $M$
the proof of Lemma \ref{lemma:basic_lemma} shows that
\beq
\A_H(x)\in\frac1N\Z\,.
\eeq
\end{Rmk}

Let $H$ be a Hamiltonian function and $c\in\R$ then we denote by $H^c(t,x):=H(t,x)+c$.

\begin{Def}
For a fixed $N$ we call a nondegenerate Hamiltonian function $H:S^1\times M\pf\R$ strongly nondegenerate if the action spectrum of $\A_H$ and $\frac1N\Z$ are disjoint.
\end{Def}

\begin{Cor}
If $H$ is strongly nondegenerate then $\HH:S^1\times E^N\pf\R$ is nondegenerate.
\end{Cor}

\begin{Cor}\label{cor:non_degeneracy_of_HH}
Let $H:S^1\times M\pf\R$ be a nondegenerate Hamiltonian function. Then there exists an arbitrarily small constant $c$ such that $H^c$
is strongly nondegenerate.
\end{Cor}

\begin{proof}
Both corollaries follow from Lemma \ref{lemma:basic_lemma} by noting that the number of critical values of $\A_H$ for nondegenerate $H$ is finite.
\end{proof}

\subsection{Convexity}\label{sec:convexity}

In this section we prove a convexity result for a class of Hamiltonian functions in the symplectization of a contact manifold $(\Sigma,\xi)$.
We assume that $\xi$ arises as the kernel of a contact form $\alpha$. Then the symplectization can be written as
$(\R\times\Sigma,\Omega:=d(e^\rho\alpha))$.
The symplectization admits the natural Liouville vector field ${X}=\frac{\partial}{\partial\rho}$ which induces the flow $\phi_{X}^t(\rho,x)=(\rho+t,x)$.
Furthermore, the Reeb vector field of $\alpha$ is denoted by $R$. We recall that it is uniquely defined by the properties $\alpha(R)=1$ and
$\iota_R\,d\alpha=0$.

We denote by $\J_\Sigma$ the space of almost complex structures $J$ on $\R\times\Sigma$ satisfying the following properties
\begin{enumerate}
\item $J$ is invariant under the Liouville flow $\phi_{X}^t$,
\item $J(\xi)=\xi$ and is compatible with the fiber-wise symplectic structure $d\alpha$ on $\xi$,
\item $J({X})=R$.
\end{enumerate}
Such a $J$ induces the Riemannian metric $g(\cdot,\cdot)=\Omega(\cdot,J\cdot)$ on $\R\times\Sigma$.
If we define the function $f\in C^\infty(\R\times\Sigma)$ by $f(\rho,x):=e^\rho$ we obtain
\beq
\nabla f=X\quad\text{and}\quad g({X},{X})=f\;.
\eeq
We define the following class of Hamiltonian functions
\beq
\mathcal{H}_\Sigma:=\{H\in C^\infty(S^1\times\Sigma)\mid dH_t(R)=0\}\;.
\eeq
For $H\in\mathcal{H}_\Sigma$ we set
\beq
\HH(t,\rho,x):=f(\rho,x)\cdot H(t,x)\;.
\eeq
We point out that the Hamiltonian functions $\HH$ as defined in Section \ref{sec:the_setting} belong to $\mathcal{H}_\Sigma$.

\begin{Rmk}
If we extend $H\in\mathcal{H}_\Sigma$ to $\R\times\Sigma$ independently of the $\R$-variable we obtain a $\phi_{X}$-invariant function
which we denote by $H$ again.
\end{Rmk}

\begin{Prop}\label{prop:convexity}
Let $U$ be some open subset of $\C$ and $H\in\mathcal{H}_\Sigma$. We consider a map $u\in C^\infty(U,\R\times\Sigma)$ solving Floer's equation
\beq
\tag{$\ast$}\partial_su+J(s,t,u)\big(\partial_tu-X_\HH(t,u)\big)=0\qquad \forall\, s+ti\in U
\eeq
for a smooth family $J(s,t)\in\J_\Sigma$. Then
\beq
\Delta\big(f(u)\big)=||\partial_su||^2\;.
\eeq
\end{Prop}

\begin{proof}
We use the identity
\beq
-dd^c\big(f(u)\big)=\Delta\big(f(u)\big)ds\wedge dt\,,
\eeq
where $d^c(f(u))=d(f(u))\circ i$. Using $\nabla f={X}$, $dH_t(R)=0$ and Floer's equation $(*)$ we compute
\bea
-d^c\big(f(u)\big)&=g({X},\partial_su)dt-g({X},\partial_tu)ds\\
&=u^*(\iota_{X}\Omega)+d\HH_t({X})dt-d\HH_t(J{X})ds\\
&=u^*(\iota_{X}\Omega)+\HH_t(u)dt\;.
\eea
Therefore, $\L_{X}\Omega=\Omega$, Cartan's formula and $(*)$ implies
\bea
\Delta\big(f(u)\big)ds\wedge dt&=-dd^c(f(u))\\
&=u^*\Omega+\frac{d}{ds}\HH_t(u)ds\wedge dt\\
&=\Big[||\partial_su||^2-d\HH_t(\partial_su)+\frac{d}{ds}\HH_t(u)\Big]ds\wedge dt\\
&=||\partial_su||^2\,ds\wedge dt\;.
\eea
\end{proof}

\begin{Rmk}
If we consider $s$-dependent families $H(s,t,x)$ where $H(s,\cdot,\cdot)\in\mathcal{H}_\Sigma$ for all $s$ then the result of the last Proposition
is modified to
\beq
\Delta\big(f(u)\big)=||\partial_su||^2+f(u)\cdot\frac{\partial H}{\partial s}(t,u)\;.
\eeq
\end{Rmk}
By standard application of the Maximum Principle (see for example
\cite[Theorem 3.5]{Gilbarg_Trudinger_Elliptic_partial_differential_equations_of_second_order}) we obtain the following
\begin{Cor}\label{cor:convexity_at_infty}
If $u\in C^\infty(U,\R\times\Sigma)$ is a solution of Floer's equation $(\ast)$ and $f\circ u$ attains a maximum then $f\circ u$ is constant.

Moreover, if we allow in Floer's equation $(\ast)$ $s$-dependent families $H(s,t,x)$ then the assertion holds under the additional assumption
\beq
\frac{\partial H}{\partial s}(t,u)\geq0\;.
\eeq
\end{Cor}

\begin{Rmk}
Assume that the Hamiltonian function $H$ in Proposition \ref{prop:convexity} is autonomous and takes only strictly positive values,
$H:\R\times\Sigma\pf(0,\infty)$. Then
$\alpha_H:=\frac{1}{NH}\alpha$ is another defining contact form for the contact structure $\xi$ on $\Sigma$. Furthermore,
solutions of Floer's equation with Hamiltonian function $\HH$ on the symplectization of $(\Sigma,\alpha)$ coincide with solutions of Floer's equation
with Hamiltonian function $\widehat{1}$
on the symplectization of $(\Sigma,\alpha_H)$. Then Proposition \ref{prop:convexity} reduces to the standard convexity results in symplectic
homology, see for instance \cite[section 2]{Floer_Hofer_Symplectic_homology_I}.
\end{Rmk}

\subsection{Index considerations}\label{sec:index_considerations}

According to equation \eqref{eqn:Ham_vfield_of_HH} a 1-periodic solution of $X_H$ can either be considered as lying in $M$ or in
(the zero-section of) $E^N$. The next proposition computes the difference $\CZ^{E^N}(x;\HH)-\CZ^M(x;H)$ of the Conley-Zehnder indices
in terms of the action value $\A_H(x)$. We denote by $\lfloor\beta\rfloor$ the integer part or Gauss bracket of a real number $\beta\in\R$.

\begin{Prop}\label{prop:index_formula}
Fix a bundle $p:E^N\pf M$ and a strongly nondegenerate Hamiltonian function $H:S^1\times M\pf\R$.
Let $x$ be a 1-periodic orbit of $X_H$ or equivalently of $X_\HH$ then
\beq
\CZ^{E^N}(x;\HH)=\CZ^M(x;H)+2\lfloor N\A_H(x)\rfloor+1\;.
\eeq
\end{Prop}

\begin{proof}
This follows directly from equation \eqref{eqn:matrix_without_star} and the product property of the Conley-Zehnder index
(see \cite[section 2.4]{Salamon_lectures_on_floer_homology}) by noting that $\CZ(e^{2\pi i\beta t},t\in[0,1])=2\lfloor\beta\rfloor+1$.
\end{proof}

\begin{Rmk}
The above index formula reflects the following symmetry breaking. If $H:M\pf(0,\infty)$ is a $C^2$-small, positive Hamiltonian function then
$-n-1\leq\CZ^{E^N}(x;\HH)\leq n-1$ whereas if $H:M\pf(-\infty,0)$ is a $C^2$-small, negative Hamiltonian function then $-n+1\leq\CZ^{E^N}(x;\HH)\leq n+1$. This is due to the fact that in the latter case $\HH$ is negative quadratic in the fiber direction whereas in the former case it is positive quadratic.
\end{Rmk}

\subsection{Definition of Floer Homology}\label{sec:def_of_FH_for_neg_bdls}

Let $(M,\om)$, $p:E^N\pf\R$, $\Omega$ and $\HH$ as in Section \ref{sec:the_setting}, where $H$ is strongly nondegenerate. In particular, the set of
1-periodic orbits $\P^1(\HH)$ is finite, hence we define $\CF_k^N(H):=\CF_k(\HH)$ as in Section \ref{sec:Floer_hom_closed_case},
graded by the Conley-Zehnder index $\CZ^{E^N}$ on $E^N$.

Let the contact hypersurface $\Sigma_c$ for some $c>\frac1N$ be defined as in equation \eqref{eqn:def_of_Sigma_c}.
We denote by $\J^N_{\mathrm{conv}}$ the space of smooth $S^1$-families of $\Omega$-compatible almost complex structures $J$ on $E^N$ with the property
that there exists a compact neighborhood $K=K(J)$ of the zero-section in $E^N$ and an $S^1$-family of almost complex structures $J'(t)\in\J(\Sigma_c)$,
$t\in S^1$ such that $J$ and $J'$ agree on $E^N\setminus K$.

Floer homology can now be defined as in Section \ref{sec:Floer_hom_closed_case}, since $(E^N,\Omega)$ is symplectically aspherical and
convex at infinity. More precisely, by Corollary \ref{cor:convexity_at_infty} all solutions of Floer's equation are contained within the compact set $K$,
compare Remark \ref{rmk:compactness}. Thus, we obtain a complex $(\CF_*^N(H),\partial^N)$.

\begin{Def}\label{def:bundle_FH}
For a strongly nondegenerate $H$ we set
\beq
\HF^N_*(H):=\H_*\big(\CF_\bullet^N(H),\partial^N\big)
\eeq
which is $\Z$-graded $\Z/2$-vector space.
\end{Def}
To extend the definition of continuation homomorphisms $m(H_1,H_0)$ from Section \ref{sec:Floer_hom_closed_case} to the current setting we need
to ensure that the convexity at infinity applies to the moduli spaces $\M(x_-,x_+;\HH_s)$ for a 1-parameter family $H_s$. According to
Corollary \ref{cor:convexity_at_infty} this is the case if $H_0(t,x)\leq H_1(t,x)$ for all $(t,x)\in S^1\times M$ since then we can choose
$H_s=(1-\beta(s))H_0+\beta(s)H_1$ for some monotone smooth cut-off function $\beta:\R\pf\R$ satisfying $\beta(s)=0$ for $s\leq0$
and $\beta(s)=1$ for $s\geq1$. We note that $H_0\leq H_1$ implies $\HH_0\leq \HH_1$. As in the compact case
$m(\HH_1,\HH_0)$ does not depend on the chosen 1-parameter family $\HH_s$ given that $\frac{\partial\HH_s}{\partial s}\geq0$ holds.
\begin{Def}\label{def:continuation_hom}
For Hamiltonian functions $H_0, H_1:S^1\times M\pf\R$ satisfying $H_0\leq H_1$ we denote the continuation homomorphism $m(\HH_1,\HH_0)$ by
\beq
m(H_1,H_0):\HF^N_*(H_0)\pf\HF^N_*(H_1)\;.
\eeq
\end{Def}

The following Proposition is proven as in the closed case, see for instance
\cite{Salamon_lectures_on_floer_homology,Floer_Hofer_Symplectic_homology_I}.

\begin{Prop}\label{prop:continuation_hom}
For Hamiltonian functions $H_0\leq H_1\leq H_2$ the following equality holds
\beq
m(H_2,H_1)\circ m(H_1,H_0)=m(H_2,H_0)\;.
\eeq
\end{Prop}
The first part of Theorem B is the following Proposition.

\begin{Prop}\label{prop:negative_bundles_detect_orbits}
For a nondegenerate Hamiltonian function $H$ there exists a negative line bundle $p:E^N\pf\R$ and an arbitrarily small constant $c$ such that
\beq
\dim\HF^N(H^c)=\#\P^1(H)
\eeq
where $\P^1(H)$ is the set of contractible 1-periodic orbits of the Hamiltonian vector field of $H$.
\end{Prop}

\begin{proof}
This is an application of the index formula in Proposition \ref{prop:index_formula}. Since $H$ is nondegenerate the set $\P^1(H)$ is finite.
Thus, we can choose an arbitrarily small $c$ such that $\A_{H^c}$ has only irrational critical values. Now we choose $N$ so large that
for all $x,y\in\P^1(H)$ with $\A_{H^c}(x)\not=\A_{H^c}(y)$
\beq
|\CZ^{E^N}(x;\HH^c)-\CZ^{E^N}(y;\HH^c)|\geq2
\eeq
holds. This is possible according to Proposition \ref{prop:index_formula}. This implies that the boundary operator in Floer's complex vanishes, since
action along gradient flow lines strictly decreases and the boundary operator is of degree $-1$.
\end{proof}

\section{Floer homology for negative line bundles - the relative case}
\label{sec:HF^E_relative_case}
\subsection{Bohr-Sommerfeld Lagrangian submanifolds}
\label{sec:Bohr-Sommerfeld}
Inspired by the work \cite{Eliashberg_Polterovich_Partially_ordered_groups_and_geometry_of_contact_transformations} by Eliashberg and Polterovich
we make the following definition.

\begin{Def}\label{def:Bohr_Sommerfeld}
Let $(M,\om)$ be a symplectic manifold with integral symplectic form, $[\om]\in\H^2(M;\Z)$, and $L\subset M$ be a Lagrangian submanifold. We call a pair
$(E,\alpha)$ consisting of a complex line bundle $p:E\pf M$ and a connection 1-form $\alpha$ on $E$ a \textit{Bohr-Sommerfeld pair}
for $(M,\om,L)$ if the following holds
\begin{enumerate}
\item $F_{\alpha}=N\om$ for some $N=N(E,\alpha)\in\N_{>0}$,
\item The holonomy homomorphism $\mathrm{hol_{\alpha|_L}}:\pi_1(L)\pf S^1$ of $(E|_L,\alpha|_L)$ takes values only in $\{0,\frac12\}\subset S^1=\R/\Z$.
\end{enumerate}
The integer $N(E,\alpha)$ is called the power of $(E,\alpha)$.
\end{Def}
\begin{Rmk}$ $
\begin{itemize}
\item Since $L$ is a Lagrangian submanifold and the curvature of $E$ equals $F_{\alpha}=N\om$, the bundle $(E|_L,\alpha|_L)$ is
flat over $L$ and thus, the holonomy homomorphism $\mathrm{hol}_{\alpha|_L}:\pi_1(L)\pf S^1$ is well-defined.
\item If $(E,\alpha)$ is a Bohr-Sommerfeld pair for $(M,\om,L)$ then so is $\big(E^{\otimes k},\alpha^{\otimes k}\big)$ for any $k\in\N$, and
$N(E^{\otimes k},\alpha^{\otimes k})=kN(E,\alpha)$. We refer the reader to Proposition \ref{prop:holonomy_appendix} in the appendix for further details.
\end{itemize}
\end{Rmk}
In the following we give two existence criteria for Bohr-Sommerfeld pairs, see Corollary \ref{cor:Bohr_Sommerfeld_existence_injective} and Theorem
\ref{thm:Bohr_Sommerfeld_involution}.
\begin{Prop}
Let $p:E\pf M$ be a complex line bundle with $c_1(E)=-[\omega]$. Furthermore, we assume that the map $i_1:\H_1(L;\R)\pf\H_1(M;\R)$ is injective
and that the bundle $E|_L\pf L$ is trivializable. Then there exists a 1-form $\alpha$ such that
$(E,\alpha)$ is a Bohr-Sommerfeld pair of power $N(E,\alpha)=1$.
\end{Prop}
\begin{proof}
First we choose a connection 1-form $\alpha$ satisfying $d\alpha=F_{\alpha}=\om$. The last equation determines $\alpha$ up to adding
$p^*\tau$ where $\tau\in\Omega^1(M)$ is closed. Since $L$ is Lagrange we conclude $F_{\alpha}|_L=0$, that is, the bundle $E|_L$ is flat, thus
\beq
\mathrm{hol_\alpha}:\pi_1(L)\pf S^1
\eeq
is defined.
Since $E|_L$ is trivializable we can choose a connection 1-form $\alpha^L$ on $E|_L$ with vanishing curvature and trivial holonomy. In particular,
\beq
\alpha|_L-\alpha^L=p^*\beta
\eeq
where $\beta\in\Omega^1(L)$. Due to the vanishing of the curvature of both connections we conclude $d\beta=0$. Since by assumption
$i^1:\H^1(M;\R)\pf\H^1(L;\R)$ is surjective, there exists $[\tau]\in\H^1(M;\R)$ such that $i^1([\tau])=[\beta]$.
By construction, we have $\beta-\tau|_L=df$ for some function $f:L\pf\R$. We extend $f$ to $\tilde{f}:M\pf\R$ and set
$\widetilde{\tau}:=\tau+d\tilde{f}$. In particular, $\beta=\widetilde{\tau}|_L$ holds. We define
\beq
\widetilde{\alpha}:=\alpha-p^*\widetilde{\tau}\,.
\eeq
We notice that $F_{i\widetilde{\alpha}}|_L=0$ and
\beq
\widetilde{\alpha}|_L=\alpha|_L-p^*\widetilde{\tau}|_L=\alpha|_L-p^*\beta=\alpha^L
\eeq
has trivial holonomy.
\end{proof}

\begin{Rmk}
In the above proof we construct a connection $\alpha$ with trivial holonomy. In particular, the bundle $E|_L\pf L$ is canonically trivialized
via the parallel transport of $\alpha$.
\end{Rmk}

\begin{Cor}\label{cor:Bohr_Sommerfeld_existence_injective}
Let $(M,\om)$ be an integral symplectic manifold, and $i:L\subset M$ a Lagrangian submanifold such that $i_1:\H_1(L;\R)\pf\H_1(M;\R)$ is injective.
We denote by $E$ the complex line bundle with $c_1(E)=-[\om]$.

\noindent (a) Then there exists an integer $N>0$ and a connection 1-form $\alpha$ such that
$(E^{\otimes N},\alpha)$ is a Bohr-Sommerfeld pair for $(M,\om,L)$ of power $N$.

\noindent (b) In case that $\H^2(L;\Z)$ is a free abelian group we can choose $N=1$.
\end{Cor}

\begin{proof}
Since $\om$ is integral we can choose $p:E\pf M$ with $c_1(E)=-[\om]$. Then $c_1(E|_L)=i^*c_1(E)=-i^*[\om]=0\in\H^2(L,\R)$ since $L$ is Lagrangian
submanifold.
In particular, $c_1(E|_L)\in\H^2(L;\Z)$ is a torsion class, and thus, there exists an integer $N$ such that $0=Nc_1(E|_L)=c_1(E^{\otimes N}|_L)$.
Therefore, $E^{\otimes N}|_L\pf L$ is trivializable  and the assertion of the corollary follows from the preceding proposition.
\end{proof}

\begin{Ex}
If $(M,\om)$ is an integral symplectic manifold then the diagonal $\Delta\subset (M\times M,(-\om)\oplus\om)$ is Bohr-Sommerfeld by the previous
corollary.
\end{Ex}

An additional source of examples of Bohr-Sommerfeld pairs are integral symplectic manifolds $(M,\om)$ supporting an anti-symplectic involution.
Following Welschinger \cite{Welschinger_Invariants_of_real_rational_symplectic_2_mfds} we say that a triple
$(M,\omega,R)$ is a \emph{real symplectic manifold} if $(M,\omega)$ is a symplectic manifold and $R \in \mathrm{Diff}(M)$ is
an anti-symplectic involution, that is
\beq\label{eqn:chern}
R^2=\mathrm{id}, \quad R^*\omega=-\omega.
\eeq
Note that the fixed point set of an anti-symplectic involution is a (maybe empty) Lagrangian submanifold of $(M,\omega)$.

\begin{Thm}\label{thm:Bohr_Sommerfeld_involution}
Let $(M,R,\omega)$ be an integral real symplectic manifold. Then there exists a Bohr-Sommerfeld pair for $(M,\omega,\mathrm{Fix}R)$.
\end{Thm}

\begin{proof}
As noted above since $(M,\omega)$ is integral there exists a complex line bundle $E_\om\pf M$ satisfying
\beq
c_1(E_\om)=-[\omega]\in\H^2(M;\R)\,.
\eeq
Since $R^*\omega=-\omega$ it follows that
\beq
c_1(R^*E_\omega)=-c_1(E_\omega)\in\H^2(M;\R)\,.
\eeq
Hence as in the proof of Corollary \ref{cor:Bohr_Sommerfeld_existence_injective} there exists $n\in\N$ such that
\beq
c_1(R^*E_\omega^{\otimes n})=-c_1(E_\omega^{\otimes n})\in\H^2(M;\Z)\,.
\eeq
We set $E:=E_\omega^{\otimes N}$ with $N:=2n$. Thus, we obtain a complex line bundle $p:E\pf M$.

We claim that the involution $R$ now extends naturally to an $S^1$-invariant involution $R^E$ of the
bundle $E$ with the property
\beq
p\circ R^E=R\circ p\,.
\eeq
This is the content of Proposition \ref{prop:involution_extends} below. Assuming this fact we complete the proof of the theorem.
We choose a connection 1-form $\alpha_0$ on $E$ satisfying
\beq
F_{\alpha_0}=-N\omega\,.
\eeq
Define a $R^E$-antiinvariant connection 1-form $\alpha$ on $E$ by
\beq
\alpha=\tfrac12\big(\alpha_0-(R^E)^*\alpha_0\big)\,.
\eeq
We note that
\beq
F_{\alpha}=d\alpha=\tfrac12\big(d\alpha_0-(R^E)^*d \alpha_0\big) =-\tfrac{N}{2}\big(\omega-(R^E)^*\omega\big)=-N\omega\,
\eeq
To compute the holonomy of $\alpha$ on the Lagrangian submanifold $\mathrm{Fix}R$ we pick $\gamma \in C^\infty(S^1,\mathrm{Fix}R)$.
Furthermore, we choose a loop $e\in C^\infty(S^1,E \setminus M)$
satisfying $p\circ e=\gamma$. Then the holonomy of $\alpha$ along $\gamma$ is given by
\beq
\mathrm{hol}_\alpha(\gamma)=-\int_0^1 e^* \alpha\,.
\eeq
Using $(R^E)^*\alpha=-\alpha$ we compute
\beq
-\int_0^1 e^* \alpha=\int_0^1e^*(R^E)^*\alpha=\int_0^1(R^E\circ e)^*\alpha\,.
\eeq
From $p\circ R^E\circ e=R\circ p\circ e=R\circ \gamma$ we conclude
\beq
\mathrm{hol}_\alpha(\gamma)=-\int_0^1 e^* \alpha=\int_0^1(R^E\circ e)^* \alpha =-\mathrm{hol}_\alpha(R\circ\gamma)
\eeq
Since $\gamma$ takes values in $\mathrm{Fix}R$ we have $R\circ\gamma=\gamma$, thus we finally conclude
\beq
\mathrm{hol}_\alpha(\gamma)=-\mathrm{hol}_\alpha(R\circ\gamma)=-\mathrm{hol}_\alpha(\gamma)
\eeq
This implies that $\mathrm{hol_\alpha}(\gamma)\in\{0,\frac12\}\subset S^1$. Hence the tuple $(E,\alpha)$ is a Bohr-Sommerfeld pair for
$(M,\omega,\mathrm{Fix}R)$.
\end{proof}

It remains to prove the following proposition.

\begin{Prop}\label{prop:involution_extends}
There exists an $S^1$-invariant involution $R^E$ of the bundle $E$ extending $R$, more precisely
\beq
p\circ R^E=R\circ p\,.
\eeq
\end{Prop}

\begin{proof} We use the same notation as in the previous proof and abbreviate by $F=E_\omega^{\otimes n}$ the square root of $E$.
It follows from (\ref{eqn:chern}) and the fact that a complex line bundle is determined up
to $C^\infty$ isomorphism by its first Chern class, see \cite[Chapter 1.1]{Griffiths_Harris_Principles_of_algebraic_geometry}, that
\beq\label{eqn:iso}
R^*F\cong F^*.
\eeq
We denote by
\beq
\mathcal{F}=\{e \in F: ||e||=1\}
\eeq
the unit circle bundle of $F$. We note that $\mathcal{F}$ is a principal $S^1$-bundle
over $M$. We denote the $S^1=\R/\Z$-action by $g.e$. Then \eqref{eqn:iso} implies (see also Appendix \ref{appendix:holonomy}) that there exists a smooth map
$\psi:\mathcal{F} \pf \mathcal{F}$ satisfying
\beq\label{eqn:psi}
\left.\begin{array}{rlc}
\psi(g.\sigma)&=\;(-g).\psi(\sigma) & \quad g \in S^1,\,\sigma \in \mathcal{F}\\[1ex]
p\circ\psi&=\;R\circ p
\end{array}\right\}
\eeq
The map $\psi$ need not be an involution.
However, it follows from \eqref{eqn:psi} and the fact that $R$ is an involution that there exists a map $\rho \in C^\infty(M,S^1)$ such that
\beq\label{eqn:inv}
\psi^2(\sigma)=\rho\big(p(\sigma)\big).\sigma, \quad \sigma \in \mathcal{F}\,.
\eeq

\begin{Lemma}
The map $\rho$ satisfies the equation
\beq\label{eqn:rho}
\rho(x)=-\rho\big(R(x)\big), \quad x \in M.
\eeq
\end{Lemma}

\begin{proof}[Proof of the Lemma]
To prove \eqref{eqn:rho} we compute $\psi^3$ in two ways. First note that for $\sigma \in \mathcal{F}$ and $x=p(\sigma)$ it follows from \eqref{eqn:psi}
\beq
\psi^3(\sigma)=\psi\big(\psi^2(\sigma)\big)=\psi(\rho(x).\sigma)= (-\rho(x)).\psi(\sigma)\,.
\eeq
Alternatively we compute using $p\circ\psi=R\circ p$ and \eqref{eqn:inv}
\beq
\psi^3(\sigma)=\psi^2\big(\psi(\sigma)\big)=\rho\big(p(\psi(\sigma))\big).\psi(\sigma)=\rho\big(R(p(\sigma))\big).\psi(\sigma)=\rho\big(R(x)\big).\psi(\sigma)
\eeq
The two equations imply \eqref{eqn:rho}.
\end{proof}

The gauge group $C^\infty(M,S^1)$ acts on solutions of \eqref{eqn:psi} in the following way. Let $\gamma \in C^\infty(M,S^1)$
and $\psi$ be a solution of \eqref{eqn:psi} then the map $\psi_\gamma:\mathcal{F} \pf \mathcal{F}$ defined by
\beq
\psi_\gamma(\sigma)=\gamma(p(\sigma)).\psi(\sigma)
\eeq
is also a solution of \eqref{eqn:inv}. Let $\rho_\gamma \in C^\infty(M,S^1)$ be as in \eqref{eqn:inv} the obstruction for $\psi_\gamma$ to be an
involution.
\begin{Lemma}
The maps $\rho, \rho_\gamma \in C^\infty(M,S^1)$ are related by
\beq\label{eqn:gamma}
\rho_\gamma(x)=\gamma\big(R(x)\big).\rho(x).(-\gamma(x)), \quad x \in M.
\eeq
\end{Lemma}
\begin{proof}[Proof of the Lemma]
For $\sigma \in \mathcal{F}$ and $x=p(\sigma) \in M$ we compute
\bea
\rho_\gamma(x).\sigma &= \psi^2_\gamma(\sigma)\\
&=\psi_\gamma\big(\gamma(x).\psi(\sigma)\big)\\
&=\Big(-\gamma(x)+\gamma\big(p(\psi(\sigma))\big)\Big).\psi^2(\sigma)\\
&=\Big(-\gamma(x)+\gamma\big(R(x)\big)\Big).\psi^2(\sigma)\\
&=\Big(-\gamma(x)+\gamma\big(R(x)\big)+\rho(x)\Big).\sigma
\eea
This implies \eqref{eqn:gamma}.
\end{proof}
If $\mathcal{F}$ is a principal $S^1$-bundle over $M$, then the tensor product of $\mathcal{F}$ with itself is given by
\beq
\mathcal{F}_2= \mathcal{F} \otimes \mathcal{F}=\big(\mathcal{F} \times _{M} \mathcal{F}\big) /{\bar{\Delta}}\,.
\eeq
Here $\mathcal{F} \times_M \mathcal{F}$ is the fiber product of $\mathcal{F}$ with itself over $M$. This is a principal torus bundle over $M$.
Dividing out the antidiagonal $\bar{\Delta}=\big\{(g,-g): g \in S^1\big\}\subset T^2$ we obtain a principal $S^1$-bundle
over $M$ again. A smooth map $\psi:\mathcal{F} \pf \mathcal{F}$ satisfying \eqref{eqn:psi} induces a map
$\psi_2:\mathcal{F}_2 \pf \mathcal{F}_2$ defined by
\beq
\psi_2[(\sigma_1,\sigma_2)]=[\psi(\sigma_1),\psi(\sigma_2)],
\quad (\sigma_1,\sigma_2) \in
\mathcal{F} \times_M \mathcal{F}.
\eeq
Note that $\psi_2$ satisfies \eqref{eqn:psi} for $\mathcal{F}_2$ again. Let $\rho_2 \in C^\infty(M,S^1)$ be the obstruction of $\psi_2$
to being an involution. Note that
\beq
\rho_2(x)=2\rho(x), \quad x \in M\,.
\eeq

\begin{Lemma}
The map $\Psi=(\psi_2\big)_\rho=\rho.\psi_2$ is an involution on $\mathcal{F}_2$.
\end{Lemma}

\begin{proof}[Proof of the Lemma]
Using \eqref{eqn:rho} and \eqref{eqn:gamma} we compute for $x \in M$
\beq
\rho_{2,\rho}(x)=\rho\big(R(x)\big)+\rho_2(x)-\rho(x)=-\rho(x)+2\rho(x)-\rho(x)=0.
\eeq
This proves the Lemma.
\end{proof}

To finish the proof of Proposition \ref{prop:involution_extends} we note that the complex line bundle $E$ is by definition
$E=F^{\otimes 2}=\big(\mathcal{F}_2 \times \mathbb{C}\big)/S^1$. We  define a involution $R^E:E \pf E$ by the formula
\beq
R^E\big[(\sigma,z)\big]=\big[(\Psi(\sigma),-z)\big]\,.
\eeq
Then the property $p\circ R^E=R\circ p$ follows immediately from the fact that $\Psi$ is a solution of \eqref{eqn:psi}.
This finishes the proof of Proposition \ref{prop:involution_extends}.
\end{proof}

\begin{Ex}
The Clifford torus $\mathbb{T}^n\in\CP^n$ is the fixed point set of an anti-symplectic involution given by $[z_0:\ldots :z_n]\mapsto [\frac{z_0}{|z_0|^2}:\ldots:\frac{z_n}{|z_n|^2}]$. Thus, Theorem \ref{thm:Bohr_Sommerfeld_involution} applies and $\mathbb{T}^n$ is a Bohr-Sommerfeld Lagrangian, even though the simpler homological condition of Corollary \ref{cor:Bohr_Sommerfeld_existence_injective} does not apply. 
\end{Ex}

\subsection{Definition of Floer Homology}
\label{sec:def_of_FH_for_neg_bdls_relative case}

We are considering an integral, closed, symplectically aspherical symplectic manifold $(M,\om)$.
Furthermore, we assume that $L\subset M$ is a closed Lagrangian submanifold which is symplectically aspherical:
\beq\label{eqn:aspherical_Lagrangian}
\Mas|_{\pi_2(M,L)}=0\quad\text{and}\quad\om|_{\pi_2(M,L)}=0\,.
\eeq
Let $p:E^N\pf M$ be a complex line bundle and $\alpha$ a connection 1-form as in Section \ref{sec:the_setting},
that is $c_1(E)=-N[\om]$. We assume that $(E^N,\alpha)$ is a Bohr-Sommerfeld pair for $(M,\om,L)$. By definition the power of $(E^N,\alpha)$ is $N$.
We fix an identification of a fiber $E^N_x\cong\C$ for some $x\in L$.
Since the holonomy of $\alpha|_L$ takes only values in $\{\pm1\}$ parallel transport along any loop in $L$ starting at $x$ will map
$\R\subset\C\cong E^N_x$ into itself. Thus, parallel transport along paths in $L$ defines a $\R$-vector bundle $L^N$ over $L$.
We obtain a non-compact Lagrangian submanifold $\LH\subset(E^N,\Omega)$ satisfying
\beq
\Mas|_{\pi_2(E^N,\LH)}=0\quad\text{and}\quad\Omega|_{\pi_2(E^N,\LH)}=0\,.
\eeq
If the holonomy is trivial then the bundle $E^N|_L$ is canonically trivialized, i.e.~$E^N|_L=L\times\C$ and $\LH=L\times\R$.

To define Lagrangian Floer homology $\HF_*^N(H;L)$ for a Hamiltonian function $\HH$ as defined in Section \ref{sec:the_setting}
we establish a relative version of Corollary \ref{cor:convexity_at_infty}. Let $u:\R\times[0,1]\pf E^N$ be a solution
of Floer's equation with Lagrangian boundary conditions. We recall that the Hamiltonian $\HH$ is of the form $Nf(r)H(t,x)$. In Proposition
\ref{prop:convexity} we derived for the function $F(s,t):=f\circ u(s,t):\R\times[0,1]\pf\R$ the equality
\beq
\Delta F=||\partial_su||^2\,.
\eeq
In order to apply the maximum principle at a boundary point of $\R\times[0,1]$ we employ a reflection argument. For
convenience we only treat the boundary component $\{0\}\times\R$. We extend $F$ to a map $F:[-1,1]\times\R\pf\R$ by
setting $F(s,-t)=F(s,t)$. In order to prove $F\in W^{2,2}_{\mathrm{loc}}([-1,1]\times\R)$ it
suffices to show $\frac{\partial F}{\partial t}(s,0)=0$ for all $s\in\R$. This relies on the following facts established earlier in
Section \ref{sec:convexity}
\begin{itemize}
\item $X(x)\in T_x\LH$ for all $x\in\LH$,
\item $\nabla f=X$,
\item $dH(R)=0$.
\end{itemize}
At the point $(s,0)$ we compute
\bea
\frac{\partial F}{\partial t}&= <\nabla f,\partial_tu>\\
    &= <X,\partial_tu> =\om(X,J\partial_tu)\\
    &=\om(X,-\partial_su+JX_H)\\
    &=-\om(X,\partial_su)+\om(X_H,JX)\\
    &=-\om(X,\partial_su)+dH(JX)\\
    &=-\om(X,\partial_su)+dH(R)\\
    &=0
\eea
The last equality follows from the fact that both $\partial_su$ and $X$ are tangent to the Lagrangian submanifold $\LH$.
The above computation implies $F\in W^{2,2}_{\mathrm{loc}}$, thus the maximum principle applies to $F:[-1,1]\times\R\pf\R$. We conclude

\begin{Cor}\label{cor:convexity_at_infty_Lagrangian_case}
If $u:\R\times[0,1]\pf E^N$ is a solution of Floer's equation and $f\circ u$ attains a maximum then $f\circ u$ is constant.

Moreover, if we allow in Floer's equation $s$-dependent families $H(s,t,x)$ then the assertion holds under the additional assumption
\beq
\frac{\partial H}{\partial s}(t,u)\geq0\;.
\eeq
\end{Cor}
We note that critical points of $\A_\HH$ project via $p:E^N\pf M$ to critical points of $\A_H$.
The analogue of Proposition \ref{prop:index_formula} in the relative case reads

\begin{Lemma}\label{lemma:basic_lemma_rel_case}$ $
\begin{enumerate}
\item Assuming that $H$ is nondegenerate, the following are equivalent.
\begin{enumerate}
\item $\HH$ is nondegenerate.
\item $\displaystyle\A_\HH(e)\not\in \frac{1}{2N}\Z\quad\forall e\in\P_{L^N}(\HH)$.
\item All Hamiltonian chords of $\HH$ are contained in the zero-section $M$ (and then are necessarily Hamiltonian chords of $H$).
\end{enumerate}$ $
\item
Moreover, if there exists a 1-periodic chord $e$ of $X_\HH$ which is \textit{not} contained in the zero-section $M$ then all chords $r\cdot e$
obtained by fiber-wise multiplication by $r\in\R$ are 1-periodic chords of $X_\HH$. In particular,
\beq
\A_\HH(e)=\A_H\big(p(e)\big)
\eeq
in the degenerate and the nondegenerate case.
\end{enumerate}
\end{Lemma}

\begin{proof}
The proof is analogous to the proof of Lemma \ref{lemma:basic_lemma}. The only modification is in the computation of the angle. In the relative
case the Bohr-Sommerfeld condition (see Definition \ref{def:Bohr_Sommerfeld}) is crucial. We denote $x(t):=p(e(t))$
We consider a Hamiltonian chord $e\in\P_{L^N}(\HH)$ and choose a path
$\gamma:[0,1]\pf L$ in $L$ such that $\gamma(0)=x(1)$ and $\gamma(1)=x(0)$. This is possible since, by definition, $[e]=0\in\pi_1(E^N,L^N)$.
We study the parallel transport along the loop $x\#\gamma$ and consider the angle $\angle\big(e(0),P^1_\gamma(e(1))\big)$. We note that
by the Bohr-Sommerfeld condition $P^1_\gamma(e(1))\in L^N_{x(0)}$. In particular, we have $\angle\big(e(0),P^1_\gamma(e(1))\big)\in\tfrac12\Z$.
On the other hand, as in the proof of Lemma \ref{lemma:basic_lemma}, the angle computes to
\beq
\angle\big(e(0),P^1_\gamma(e(1))\big)=N\A_H(x)\in\tfrac12\Z\,.
\eeq
The remaining assertions follow as in the proof of Lemma \ref{lemma:basic_lemma}.
\end{proof}

\begin{Prop}\label{prop:index_formula_Lagrangian_case}
For $x\in\P^1_L(\HH)$
\beq
\Mas^{L^N}(x;\HH)=\Mas^L(x;H)+\lfloor N\A_H(x)\rfloor+\frac12\;.
\eeq
\end{Prop}

\begin{proof}
This follows as in the proof of Proposition \ref{prop:index_formula} and Lemma \ref{lemma:basic_lemma_rel_case}.
\end{proof}

\begin{Rmk}\label{rmk:index_formula_Lagrangian_case}
If $H:M\pf(0,\infty)$ is a $C^2$-small, positive, and strongly nondegenerate Hamiltonian function then
\beq
-\frac{n+1}{2}\leq\Mas^{L^N}(x;\HH)\leq\frac{n-1}{2}
\eeq
whereas if $H:M\pf(-\infty,0)$ is a $C^2$-small, negative, and strongly nondegenerate Hamiltonian function then
\beq
-\frac{n-1}{2}\leq\Mas^{L^N}(x;\HH)\leq\frac{n+1}{2}\,.
\eeq
\end{Rmk}

We call a Hamiltonian function $H\in C^{\infty}(I\times M)$ \textit{strongly nondegenerate} if $H$ is nondegenerate
and
\beq\label{eqn:basic_lemma_Lagr}
\A_\HH(e)\not\in \frac{1}{2N}\Z\quad\forall e\in\P_{L^N}(\HH)\;.
\eeq
In particular, as in explained in Lemma \ref{lemma:basic_lemma_rel_case}, if $\HH$ is nondegenerate, $\P^1_L(\HH)$ is a finite set and moreover,
the critical points of $\A_\HH$ and $\A_H$ coincide.

\begin{Def}
For a strongly nondegenerate $H$ we set
\beq
\HF^N_*(H;L):=\H_*\big(\CF_\bullet^N(H;\LH),\partial^N\big)\;.
\eeq
\end{Def}

We recall that in the absolute case $\HF^N_*(H)$ has been defined in Section \ref{sec:def_of_FH_for_neg_bdls}.

\section{Applications to Hamiltonian chords}\label{sec:applications}

\noindent In this section we continue to assume that $(M^{2n},\om)$ is a closed, connected, integral symplectic manifold and $L\subset M$ a closed,
connected, symplectically aspherical Lagrangian submanifold. Moreover, we assume that $(E^N,\alpha)$ is a Bohr-Sommerfeld pair for $(M,\om,L)$.

We recall from the introduction.

\begin{Def}\label{def:Ham_chord_and_general_action_functional}
Let $H:\R_+\times M\pf\R$ be a Hamiltonian function.
A pair $(x,\tau)$ where $\tau>0$ and $x\in C^\infty([0,\tau],M)$ solving
\beq\begin{cases}
&\dot{x}(t)=X_H(t,x(t))\\
&x(0),x(\tau)\in L
\end{cases}\eeq
is called a \emph{Hamiltonian chord} of period $\tau$. If the Hamiltonian chord is contractible (relative $L$) its action is defined as
\beq
\A_H(x,\tau):=-\int_{\D^2_+}\bar{x}^*\om-\int_0^\tau H(t,x(t))dt\,,
\eeq
where $\bar{x}$ realizes the homotopy of $x$ to a constant.
If $\tau=1$ we recover the previous definition $\A_H(x,1)=\A_H(x)$. The set of contractible (relative to $L$) Hamiltonian chords is denoted by $\CC(H)$.
\end{Def}

\noindent From now on we will only consider contractible Hamiltonian chords.

\begin{Rmk}
\begin{enumerate}
\item If $(x,\tau)$ is a Hamiltonian chord for $H$ then also for $H+c$ for any constant $c\in\R$.
Furthermore,
\beq\label{eqn:transformation_of_action}
\A_{H+c}(x,\tau)=\A_{H}(x,\tau)-c\tau
\eeq
\item An autonomous Hamiltonian function $H:M\pf\R$ is constant along its chords. Furthermore, for each $a\in\R_+$ there exists a
canonical map $\CC(H)\pf\CC(aH)$ given by $a*(x(t),\tau):=(x(at),\tau/a)$. The action transforms as follows
\beq\label{eqn:transformation_of_action_2}
\A_{H}(x,\tau)=\A_{aH}(a*(x,\tau))\,.
\eeq
\end{enumerate}
\end{Rmk}
From now on we only consider autonomous Hamiltonian functions.
\begin{Def}\label{def:wiggliness}
For a strongly nondegenerate Hamiltonian function $H:M\pf\R$ we define its \textit{wiggliness} to be
the minimal integer $\mathcal{W}(H)>0$ such that $\forall c\in\R$ and $\forall N\geq\mathcal{W}(H)$
the following holds
\beq
|\Mas^{L^N}(\xi;\HH_c)-\Mas^{L^N}(x;\HH_c)|\geq 2\qquad\forall x,\xi\in\P_{L}(H)\text{ with }\A_H(x)\not=\A_H(\xi)
\eeq
where $\HH_c:=\widehat{H+c}$.
\end{Def}

\begin{Lemma}\label{lemma:wiggliness_constant}
The wiggliness $\mathcal{W}(H) $ is finite and satisfies $\mathcal{W}(H)=\mathcal{W}(H+c)$.
\end{Lemma}

\begin{proof}
We first recall the following inequalities:
\bea
\lfloor a-b \rfloor + \lfloor b \rfloor = \lfloor a -b + \lfloor b \rfloor \rfloor &\geq \lfloor a-1 \rfloor= \lfloor a \rfloor -1\\
\lfloor a-b \rfloor + \lfloor b \rfloor &\leq a-b+b=a
\eea
from which we obtain the following string of inequalities:
\beq
\lfloor a \rfloor-\lfloor b\rfloor-1\leq\lfloor a - b\rfloor\leq\lfloor a\rfloor-\lfloor b\rfloor 
\eeq
Since the set $\P_L(H)$ is finite the following quantities are well-defined
\bea
\mu&:=\max\Big\{ |\Mas^L(x,H)-\Mas^L(\xi,H)|\;\Big|\; x,\xi\in\P_L(H)\Big\}\\[1ex]
\alpha&:=\min \Big\{ |\A_H(x)-\A_H(\xi)|\;\Big|\; x,\xi\in\P_L(H)\text{ with }\A_H(x)\neq\A_H(\xi)\Big\}\;.
\eea
Obviously $\alpha>0$, therefore there exists $N_0\in\N$ with
\beq
\alpha N_0\geq\mu+5\;.
\eeq
We estimate for $N\geq N_0$ and $x,\xi\in\P_L(H)\text{ with }\A_H(x)\neq\A_H(\xi)$ using Proposition \ref{prop:index_formula_Lagrangian_case}
\bea
\Big|\Mas^{L^N}(\xi;\HH_c)-\Mas^{L^N}(x;\HH_c) \Big|=&\; \Big|\Mas^L(\xi;H+c)+\lfloor N\A_{H+c}(\xi)\rfloor+\tfrac12\\ 
&\qquad-\big(\Mas^L(x;H+c)+\lfloor N\A_{H+c}(x)\rfloor+\tfrac12\big) \Big|\\[1ex]
\geq&\;-\Big|\Mas^L(\xi;H)-\Mas^L(x;H) \Big|\\
&\qquad+ \Big|\lfloor N\A_H(\xi)-Nc\rfloor-\lfloor N\A_H(x)-Nc\rfloor \Big|\\[1ex]
\geq&\;-\mu+ \Big|\lfloor N\A_H(\xi)-N\A_H(x)\rfloor \Big|-1\\[1ex]
\geq&\;-\mu+\lfloor N\alpha\rfloor-2\\[1ex]
\geq&\;-\mu+\lfloor N_0\alpha\rfloor-2\\[1ex]
\geq&\;-\mu+N_0\alpha-3\\[1ex]
\geq&\;2
\eea
Thus, the wiggliness $\mathcal{W}(H)\leq N_0$, thus finite. The second assertion is obvious from the definition.
\end{proof}

\begin{Def}\label{def:huge_Hamiltonian}
A nondegenerate positive Hamiltonian function $H:M\pf(0,\infty)$ is called \textit{huge} if it satisfies
\beq
\Mas^L(x;H)+\lfloor N\A_H(x)\rfloor<-\frac{n+1}{2},\quad\forall x\in\P_L^1(H),\;\forall N\geq \mathcal{W}(H),
\eeq
where $\dim L=n$.
\end{Def}

\begin{Rmk}\label{rmk:huge_after_adding_constant}
We note that every nondegenerate Hamiltonian function becomes huge after adding a sufficiently positive constant. Indeed, observe that the Maslov index and the wiggliness do not change under $H\mapsto H+c$ for $c>0$ whereas $\A_{H+c}=\A_H-c$. Moreover a huge function remains huge under adding positive constants.
\end{Rmk}

\begin{Prop}\label{prop:huge_computes_homology}
Let $H$ be huge and choose $N\geq\mathcal{W}(H)$. If $H$ is strongly nondegenerate for $E^N$ then the following holds
\beq
\dim\HF^N(H;L)=\#\P_L^1(H)\,,
\eeq
furthermore,
\beq\label{eqn:trivial_homology}
\HF^N_k(H;L)=0\quad\forall\,k\geq-\frac{n+1}{2}\,.
\eeq
\end{Prop}

\begin{proof}
The latter two requirements in Definition \ref{def:huge_Hamiltonian} together with the index formula from Proposition
\ref{prop:index_formula_Lagrangian_case} imply
\beq
\Mas^{L^N}(x;\HH)<-\frac{n+1}{2}
\eeq
thus \eqref{eqn:trivial_homology} follows. The first statement of the proposition follows from the fact that the
differential in Floer's complex $\CF^N(H;L)$ vanishes. Indeed, two Hamiltonian chords of $\HH$ either have the same action or Maslov index
difference different from 1.
\end{proof}

\begin{Prop}\label{prop:tiny_Hamiltonian_and_Morse}
Let $H:M\pf\R$ be a Hamiltonian function such that $H|_L:L\pf\R$ is a Morse function. Then for all $N\in\N$ there exists $\epsilon_0=\epsilon_0(N)>0$ such that for all $0<\epsilon<\epsilon_0$ there is a 1-1-correspondence between $\P_L(\epsilon H)\stackrel{1-1}{=}\Crit(H)$. Moreover, all
Hamiltonian chords $x\in\P_L(\epsilon H)$ are nondegenerate and
\beq\label{eqn:Morse_equal_Maslov_for_small}
\Mas(x)=\Morse(\hat{x})-\tfrac12\dim L
\eeq
where the critical point $\hat{x}$ of $H|_L$ corresponds to the Hamiltonian chord $x$.
If $H$ takes only positive values then $-\frac{1}{N}<\A_{\epsilon H}(x)<0$.
\end{Prop}

\begin{proof}
Let $p$ be a critical point of $H|_L$. In case that $dH(p)=0$ we are done, otherwise there exists a coordinate chart $\chi:V\pf U\subset\R^{2n}$
with the following properties, where the coordinates on $\R^{2n}$ are denoted by
  $(x_1,\ldots,x_n,y_1,\ldots,y_n)$.
\begin{itemize}
  \item $\chi(p)=0$ and $\exists a_i\neq0$ such that
    \beq
    \chi(L\cap V)=\big\{(x_1,\ldots,x_n,\frac{a_1}{2}x_1^2,\ldots,\frac{a_n}{2}x_n^2)\mid (x_1,\ldots,x_n)\in\R^n\cap U\big\}
    \eeq
  \item $\chi_*(X_H)=\sum_{i=1}^nb_i\frac{\partial}{\partial x_i}$ where the $b_i$ are constants.
\end{itemize}
The unique (and nondegenerate) chord
$(x^\epsilon,y^\epsilon):=(x_1^\epsilon(t),\ldots,x_n^\epsilon(t),y_1^\epsilon(t),\ldots,y_n^\epsilon(t))$ of period $\epsilon$ is given by
\beq
x_i^\epsilon(t)=-\frac{\epsilon b_i}{2}+b_it,\quad y_i^\epsilon(t)=\frac{a_i\epsilon^2b_i^2}{8},\quad 0\leq t\leq\epsilon\,.
\eeq

We first prove equation \eqref{eqn:Morse_equal_Maslov_for_small}. In a Weinstein neighborhood of $L$ we write $H=H\circ\pi+h$ where $\pi:T^*L\pf L$ is the projection. The equality of the Maslov index and the Morse index can be seen by choosing a homotopy from $H=H\circ\pi+h$ to $H\circ\pi$ and noting that the Hamiltonian chords of $\epsilon H\circ\pi$ are exactly the critical points for the Morse function $H|_L$.

Now assume in addition that $H$ takes only positive values. We set $c:=H(p)>0$ and denote
$\xi^\epsilon(t):=\chi^{-1}(x^\epsilon,y^\epsilon)$.
Then the above formulas imply
\beq
H(\xi^\epsilon)=c+O(\epsilon^2)
\eeq
and since $0\leq t\leq\epsilon$
\beq
\int\om(\bar{\xi^\epsilon})=O(|x^\epsilon|\cdot|y^\epsilon|)=O(\epsilon^3)
\eeq
thus
\beq\label{eqn:prop_tiny_Hamiltonian_and_Morse}
\A_{\epsilon H}(\xi^\epsilon)=-\int\om(\bar{\xi^\epsilon})-\epsilon H(\xi^\epsilon)=-\epsilon c+O(\epsilon^3)\,.
\eeq
Hence, for sufficiently small $\epsilon>0$, the action will satisfy the claimed inequality.
\end{proof}

\begin{Def}\label{def:magnitude}
Let $H:M\pf(0,\infty)$ be positive, strongly nondegenerate, and such that $H|_L:L\pf\R$ is a Morse function.
For $N\geq\mathcal{W}(H)$ we define the \textit{magnitude} of $H$ to be
\beq
\mathfrak{m}(H,N):=\inf\left\{\;r>0\;\left|\;
            \begin{aligned}&\tfrac1r H \text{ is strongly nondegenerate and }\\
            &-\frac{n+1}{2}\leq\Mas^{L^{N}}(x;\tfrac1r \widehat{H})\leq\frac{n-1}{2},\;\forall x\in\P_{L}(\tfrac1r H)
            \end{aligned}\right.\right\}
\eeq
\end{Def}

\begin{Prop}
Let $H$ be as in the definition above. Then $\mathfrak{m}(H,N)$ is finite.
\end{Prop}

\begin{proof}
This follows immediately from Propositions \ref{prop:tiny_Hamiltonian_and_Morse} and \ref{prop:index_formula_Lagrangian_case} and
Remark \ref{rmk:index_formula_Lagrangian_case}.
\end{proof}

\begin{Rmk}
If the Hamiltonian function $H$ is huge, then $\mathfrak{m}(H,N)>1$ for all $N\geq\mathcal{W}(H)$.
\end{Rmk}

\begin{Def}
Let $H$ be an autonomous Hamiltonian function and $N\in\N$. A solution $(x,\tau)$ with $\tau>0$ of
\beq
\begin{cases}
&\dot{x}=X_H(x),\;x(0),\,x(\tau)\in L,\\
&\A_H(x,\tau)\in\frac{1}{2N}\Z\,,
\end{cases}
\eeq
is called a \textit{N-quantized Hamiltonian chord}. We denote
the set of $N$-quantized chords with period less or equal than $\tau_0$ by $\P_L^\mathfrak{q}(H;\tau_0,N)$.
\end{Def}

The interest in quantized Hamiltonian chords comes from the relation to Reeb chords which we explain next. We recall that $(E^N,\alpha)$
is a Bohr-Sommerfeld pair for $(M,\om,L)$.

\begin{Lemma}\label{lemma:lifted_Lagrangian_intersects_to_Legendrian}
We denote by $(\widetilde{\Sigma}^N,\widetilde{\xi})$ the contact manifold obtained from the $S^1$-bundle of $E^N$ together with its horizontal plane
field distribution induced by $\alpha$. Then $\widetilde{\L}^N:=L^N\cap\widetilde{\Sigma}^N$ is a Legendrian submanifold of
$(\widetilde{\Sigma}^N,\widetilde{\xi})$.
\end{Lemma}

\begin{proof}
Using the contact form $\alpha$ we decompose $T_eE^N=T_e^hE^N\oplus T_e^vE^N$ into horizontal and vertical part. Then vertical part $T_e^vE^N$ is spanned
by the vectors $R_e$ and $X_e$ where $R$ is the infinitesimal generator of the $S^1$-action and $X$ the Liouville vector field, see Section
\ref{sec:the_setting}. Using the canonical identification of $T_e^hE^N\cong T_{p(e)}M$ the definition of $L^N$ immediately implies
\beq
T_eL^N\cong T_{p(e)}L\oplus \R X_e
\eeq
and thus
\beq
T_e\widetilde{\L}^N\cong T_{p(e)}L
\eeq
implies the claim.
\end{proof}

\begin{Rmk}
If the first Stiefel-Whitney class $\mathrm{w}_1(L^N)\in\H^1(L;\Z/2)$ of $L^N$ vanishes then the Legendrian submanifold
$\widetilde{\L}$ has two connected components, otherwise one.
\end{Rmk}

The group $\Z/2$ acts on $(\widetilde{\Sigma},\widetilde{\xi},\widetilde{\L})$ by $e\mapsto -e$. The quotient is denoted by $(\Sigma,\xi,L)$.
In particular, $\L$ is diffeomorphic to $L$.

As explained above Lemma \ref{lemma:Hamiltonian_vfield_equals_Reeb_vfield} (where $\widetilde{\Sigma}$ is denoted by $\Sigma$ etc.)
every positive, autonomous Hamiltonian function $H\in C^\infty(M)$ gives rise to a $S^1$-invariant contact form $\alpha_H=\frac{1}{NH}\alpha$
on $\widetilde{\Sigma}$ inducing the same contact structure $\widetilde{\xi}$. Since $\alpha_H$ is $S^1$-invariant it descends to a contact form
on $(\Sigma,\xi)$ which we denote by $\alpha_H$ again.

\begin{Prop}\label{pro:quantized_chords_are_Reeb_chords}
If the Hamiltonian function $H:M\pf(0,\infty)$ is autonomous and positive then $N$-quantized Hamiltonian chords are in 1-1-correspondence to Reeb chords of $(\Sigma,\xi,L)$ with respect to the contact form $\alpha_H$.
\end{Prop}

\begin{proof}
Lemma \ref{lemma:lifted_Lagrangian_intersects_to_Legendrian} states that $\widetilde{\L}^N=L^N\cap\widetilde{\Sigma}^N$ is a Legendrian submanifold
of $(\widetilde{\Sigma},\widetilde{\xi})$. The previous Lemma together with Lemma \ref{lemma:Hamiltonian_vfield_equals_Reeb_vfield} implies
the assertion as follows.
Given an $N$-quantized chord $x$ of $H$ we concluded in Lemma \ref{lemma:basic_lemma} resp.~Lemma \ref{lemma:basic_lemma_rel_case} that the
fibers over $x$ are filled by chords of $X_\HH$. Thus, we find a chord $e$ lying in $\Sigma$. By
Lemma \ref{lemma:Hamiltonian_vfield_equals_Reeb_vfield} the chord $e$ is a Reeb chord of $\L^N$ with respect to $\alpha_H$. Replacing $e$ by $-e$
gives rise to a different Reeb chord lying over the same quantized Hamiltonian chord $p(e)$ in $M$. After dividing out this action the
statement of the proposition follows.
\end{proof}

\begin{Def}\label{def:non_resonant}
We call an autonomous Hamiltonian function $H:M\pf\R$ \textit{non-resonant} if it satisfies for all $N\geq\mathcal{W}(H)$

\begin{enumerate}
\item $H$ is strongly nondegenerate for $N$.
\item For all $N$-quantized chords $(x,\tau)$ the following is true: $D\varphi_H^\tau(T_{x(0)}L)\pitchfork T_{x(\tau)}L$
and $X_H(x(0))\not\in T_{x(0)}L$ and $X_H(x(\tau))\not\in T_{x(\tau)}L$
\item $H|_L:L\pf\R$ is Morse.
\end{enumerate}
\end{Def}

\begin{Rmk}
Part (2) in the previous definition implies that $\Crit(H)\cap L=\emptyset$.
\end{Rmk}

The following lemma provides a sufficient condition for the number of $N$-quantized chords to be finite. We point out
that we assume do not assume the Bohr-Sommerfeld condition for $L$.

\begin{Lemma}\label{lemma:finitely_many_quantized_chords}
We assume that $L\subset (M,\om)$ is a closed, aspherical Lagrangian submanifold and that $H:M\pf (0,\infty)$ is a positive Hamiltonian function
satisfying the transversality conditions $D\varphi_H^\tau(T_{x(0)}L)\pitchfork T_{x(\tau)}L$,
$X_H(x(0))\not\in T_{x(0)}L$, and $X_H(x(\tau))\not\in T_{x(\tau)}L$ for all $N$-quantized chords.
Then the set $\P_L^\mathfrak{q}(H;\tau_0,N)$ of $N$-quantized chords with period less or equal than $\tau_0$ is finite.
\end{Lemma}

\begin{proof}
The proof of this lemma is contained in the appendix, where this is Lemma \ref{lemma:finitely_many_quantized_chords_APPENDIX}.
\end{proof}

\begin{Rmk}
We point out that a priori condition (2) in the non-resonancy definition implies that a quantized chord $(x,\tau)$ is isolated only in the set
of $\tau$-periodic chords and not necessarily in the set of all chords. The latter assertion is provided by the previous Lemma under the additional
assumptions that the Hamiltonian function $H$ satisfies $X_H(x(0))\not\in T_{x(0)}L$, $X_H(x(\tau))\not\in T_{x(\tau)}L$, and $H>0$.
Without the assumption $H>0$ quantized chords need not be
isolated. We give a counterexample in the appendix, see Example \ref{ex:counterexample_HZ}.
\end{Rmk}

\begin{Thm}\label{thm:non_resonant_is_generic}
If $\dim M\geq4$, then the set of non-resonant Hamiltonian functions is a generic subset of the set of autonomous Hamiltonian functions.
\end{Thm}
\noindent The proof is postponed to Appendix \ref{appendix:non_resonant}. \\

\noindent Let $H:S^1\times M\pf\R$ be a nondegenerate Hamiltonian function. We set
\bea
a_{\min}(H)&:=\min\{\A_H(x)\mid x\in\P_L^1(H)\}\,,\\
a_{\max}(H)&:=\max\{\A_H(x)\mid x\in\P_L^1(H)\}\,.
\eea

\begin{Thm}\label{thm:main_theorem}
We consider a non-resonant, huge Hamiltonian function $H:M\pf\R$. We choose $N\geq\mathcal{W}(H)$. Then the number of Hamiltonian chords
$(x,\tau)\in\CC(H)$ satisfying
\begin{enumerate}
\item $(x,\tau)$ is an $N$-quantized Hamiltonian chord of $H$,\\
\item $\displaystyle \frac{1}{\mathfrak{m}(H,N)}<\tau<1$.\\
\item $\displaystyle \A_H(x,\tau)\in\big[\,a_{\min}(H)-||H||,a_{\max}(H)+\max H\,\big]$
\end{enumerate}
is as least as big as $\lceil\frac12\#\P_L^1(H)\rceil$.
\end{Thm}

\noindent\textbf{Proof of Theorem A:}\quad Theorem A from the introduction is a special case of Theorem \ref{thm:main_theorem} since by  Proposition \ref{pro:quantized_chords_are_Reeb_chords} Reeb chords are in 1-1 correspondence to quantized Hamiltonian chords. Indeed the constant $N(\mathscr{D}(H))$ is the wiggliness $\mathcal{W}(H)$ and $C=C(\mathscr{D}(H))$ is such that $H+C$ is huge, see Remark \ref{rmk:huge_after_adding_constant}.

\qed

\begin{proof}[Proof of Theorem \ref{thm:main_theorem}]
We abbreviate $\wp(H):=\#\P_L^1(H)$.
Proposition \ref{prop:huge_computes_homology} implies
that there exist distinct, non-trivial class $\xi_1,\ldots,\xi_{\wp(H)}\neq0\in\HF^N_k(H)$ with $k<-\frac{n}{2}$.

We abbreviate $\mathfrak{m}:=\mathfrak{m}(H,N)$ and fix $0<\epsilon\leq\frac{1}{\mathfrak{m}}$ such that $\epsilon H$ is strongly nondegenerate and
\beq\label{eqn:small_Maslov_index}
-\frac{n+1}{2}\leq\Mas^{L^{N}}(x;\epsilon \widehat{H})\leq\frac{n-1}{2}
\eeq
for all $x\in\P(\epsilon H)$.
We recall that the set of $N$-quantized chords with period less or equal than $\tau_0$ is denoted by $\P_L^\mathfrak{q}(H;\tau_0,N)$.
We choose a function $g:\R_{\geq0}\pf\R_{\geq0}$ satisfying
\begin{enumerate}
\item $g'(\rho)=\epsilon$ for $\rho\leq\max H+\delta$ where $\epsilon$ is chosen as above,
\item $g(\rho)=\rho$ for $\rho\geq\max H+2\delta$,
\item $g''(\rho)>0$ for all $\rho\in(\max H+\delta,\max H+2\delta)$.
\end{enumerate}

where $\displaystyle 0<2\delta<\min\left\{\frac{\min(H)}{\mathfrak{m}},\frac{||H||}{\mathfrak{m}-1}\right\}$.
We note that (1) and (3) imply that $g'(\rho)$ is injective on the interval $[\max H+\delta,\max H+2\delta]$.

\begin{figure}[htb]
\input{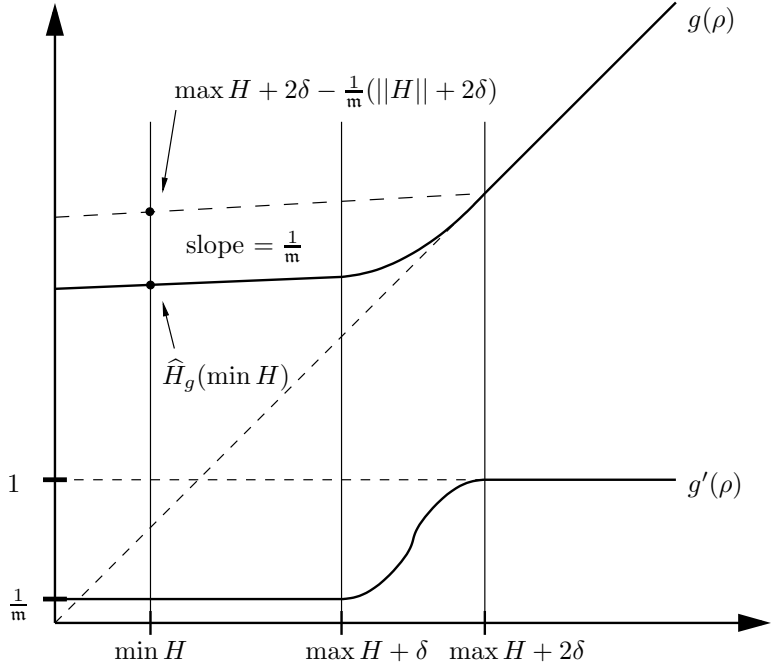}
\caption{The graph of $g$}\label{fig:graph_g}
\end{figure}

\begin{Lemma}\label{lemma:action_fctl_A_g(HH)_is_Morse}
The action functional $\A_{g(\HH)}$ is Morse.
\end{Lemma}

\begin{proof}[Proof of the Lemma]
Let $x$ be a critical point of $\A_{\HH_g}$. We have to show that $D\phi_{\HH_g}^1(T_{x(0)}L^N)\pitchfork T_{x(1)}L^N$.
If $x(0)\in L^N\setminus L$ then this follows from Proposition \ref{prop:implying_Morse_for_A_g(H)}, together with property (3)
of the function $g$ and property (2) in the non-resonancy condition of $H$ and the fact that $H$ takes only positive values, since it is huge.
If $x(0)\in L$ then by property (1) of the function $g$ and the choice of $\epsilon$ we conclude that
$(x,\tau)\in\P_L(\epsilon H)$. Since near the zero-section $M$ the function $g(\HH)$ and $\epsilon H$ agree and the latter function is
strongly nondegenerate, the lemma follows.
\end{proof}
We set $\HH_g:=g(\HH)$. The Hamiltonian vector fields transform as follows
\beq\label{eqn:transformation_of_Hamiltonian_vf}
X_{\HH_g}(e)=g'(\HH(e))\cdot X_\HH(e)\,.
\eeq
Furthermore, $g'(\HH(e))=\mathrm{const}$, according to Remark \ref{rmk:basic_rmk}. For a chord $e\in\P_{L^N}(\HH_g)$ we abbreviate
\beq
\epsilon\leq\tau_e:=g'(\HH(e))\leq1\,.
\eeq
By equation \eqref{eqn:transformation_of_Hamiltonian_vf} a chord
$e\in\P_{L^N}(\HH_g)$ is also an element $e\in\P_{L^N}(\tau_e\HH)=\P_{L^N}(\widehat{\tau_e H})$.

For all chords $e\in\P_{L^N}(\HH_g)$ with the property $\HH(e)\leq\max H+\delta$
we conclude $\tau_e=\epsilon$, by property (1) in the definition of the function $g$.
Using the equation \eqref{eqn:small_Maslov_index} and the fact that
$X_{\HH_g}=X_{\widehat{\epsilon H}}$ in
the neighborhood $\{y\in E\mid \HH(y)\leq\max H+\delta\}$ of the zero section $M$ we compute
\beq\label{eqn:Maslov_index_for_zero_section}
-\frac{n+1}{2}\leq\Mas^{L^N}(e;\HH_g)=\Mas^{L^N}(e;\widehat{\epsilon H})\leq\frac{n-1}{2}\,.
\eeq
By the second property of $g$ and equation \eqref{eqn:transformation_of_Hamiltonian_vf} we have
\beq
X_{\HH_g}(y)=X_\HH(y) \text{ for }\{y\in E\mid \HH(y)\geq\max H+2\delta\}\,.
\eeq
From the definition of $\HH$ and the assumption that $H$ takes only positive values it follows that the complement of the set
$\{y\in E\mid \HH(y)\geq\max H+2\delta\}$ has compact closure in $E$. Thus, the function $\HH_g$ is a compact perturbation of $\HH$.
The standard invariance arguments in Floer homology imply
\beq
\HF_*(\HH_g;L^N)\cong\HF_*(\HH;L^N)=\HF^N_*(H;L)\;.
\eeq
In particular, since $\xi_1,\ldots,\xi_{\wp(H)}\neq0\in\HF^N_k(H;L)$ for $k<-\frac{n+1}{2}$ we conclude that there exist distinct elements
$e_1,\ldots,e_{\wp(H)}\in\P_{L^N}(\HH_g)$ with Maslov index $\Mas(e_i;\HH_g)<-\frac{n+1}{2}$.
Therefore, by equation \eqref{eqn:Maslov_index_for_zero_section}, $e_i$ cannot lie in the neighborhood
$\{y\in E\mid \HH(y)\leq\max H+\delta\}$ of the zero section $M$.
For brevity we set $\tau_i:=\tau_{e_i}$.

We claim that $\epsilon<\tau_i<1$. By properties (1) and (3) the inequality $\epsilon<\tau_i$ follows immediately.
By definition we have $\tau_i\leq1$. In case $\tau_i=1$ we conclude from equation \eqref{eqn:transformation_of_Hamiltonian_vf} that $X_{\HH_g}(e)=X_\HH(e)$.
But by assumption $H$ is non-resonant, in particular strongly nondegenerate, thus there are no Hamiltonian chords of $X_\HH$ not lying
in the zero-section $M$.

The analog of Remark \ref{rmk:basic_rmk} in the relative case shows
\beq\label{eqn:liebling_action}
\A_{\tau_i H}(p(e_i))\in\frac{1}{2N}\Z\,.
\eeq
We set
\beq
x_i(t)=p\big(e_i(t/\tau_i)\big)
\eeq
and note that $x_i$ is a Hamiltonian chord of $H$ and has period $\tau_i$, that is $(x_i,\tau_i)\in\CC(H)$. Equation \eqref{eqn:liebling_action} and
the transformation formula \eqref{eqn:transformation_of_action_2} imply
\beq
\A_H(x_i,\tau_i)\in\frac{1}{2N}\Z\,.
\eeq

Thus, we find distinct elements $e_1,\ldots,e_{\wp(H)}$ projecting to $N$-quantized chords $(x_i,\tau_i)$, $i=1,\ldots,\wp(H)$. We claim
that not more than two $e_i$ project to the same $N$-quantized chord. In particular,
the number of $N$-quantized Hamiltonian chords equals $\big\lceil\tfrac12\wp(H)\big\rceil=\big\lceil\tfrac12\#\P_L^1(H)\big\rceil$, by definition
of $\wp(H)$.

To prove this claim we assume that there exist $e$ and $e'$ such that $\tau_e=\tau_{e'}=:\tau$ and $p\big(e(t/\tau)\big)=p\big(e'(t/\tau)\big)=:x(t)$.
We recall from the claim above that $\epsilon<\tau<1$. Since $g'(\rho)$ is injective on the interval $[\max H+\delta,\max H+2\delta]$
the equality $\tau_e=g'(\HH(e))=g'(\HH(e'))=\tau_{e'}$ implies
$\HH(e)=\HH(e')$. Since $H$ is constant along $x$ we conclude $H(p(e))=H(p(e'))$.
This implies that $f(r(e))=f(r(e'))$. Now, since $f$ is injective $r(e)=r(e')$. This proves $e=\pm e'$.\\

It remains to show that $\A_H(x,\tau)\in\big[\,a_{\min}(H)-||H||,a_{\max}(H)+\max H\,\big]$ for the $N$-quantized chords found above. This is done
in two steps.

Let $\xi\in\P_L^1(H)$ be a 1-periodic chord of $H$. Then by Proposition \ref{prop:huge_computes_homology} the chord $\xi$ defines a non-vanishing
homology class $[\xi]\in\HF^N(H;L)$. Its image under the continuation isomorphism $m(\HH_g,\HH):\HF^N(H;L)\pf\HF(\HH_g;L^N)$ can be
represented as a formal sum $\sum_{i=1}^{k^\xi}[y_i^\xi]$ where $y_i^\xi\in\P_L(\HH_g)$.

We first estimate the action value of $y_i^\xi$ from above in terms of the action value of $\xi$.
For this we interpolate between $\HH$ and $\HH_g$ via the homotopy $K_s:=\beta(s)\HH+(1-\beta(s))\HH_g$, where $\beta(s):\R\pf[0,1]$ is a
smooth monotone cut-off
function satisfying $\beta(s)=1$ for $s\leq0$ and $\beta(s)=0$ for $s\geq1$. According to Lemma \ref{lemma:energy_estimate_for_continuation}
we have the following inequality for $u\in\M(\xi,y_i^\xi;K_s)$

\bean
\A_{\HH}(\xi)-\A_{\HH_g}(y_i^\xi)&=\int_{-\infty}^\infty\int_0^1\beta'(s)(\HH-\HH_g)(u)dtds+\int_{-\infty}^\infty\int_0^1|\partial_s u|^2dtds\\
&\geq\sup(\HH-\HH_g)\underbrace{\int_{-\infty}^\infty\beta'(s)ds}_{=1}\\
&=-\sup(\HH-\HH_g)\\
&=\inf(\HH_g-\HH)=0
\eea

For simplicity we abbreviate $y=y_i^\xi$ and denote the induced quantized chord by $x(t):=p(y(t/\tau))$, where $\tau=g'(\HH(y))$. We want to find
an upper bound on the action value $\A_H(x,\tau)$ in terms of $\A_{\HH_g}(y)$. This is achieved as follows.

\bean
\A_H(x,\tau)-\A_{\HH_g}(y)&=\A_{\tau H}(\tau\ast(x,\tau))-\A_{\HH_g}(y)\\
    &=\A_{\tau H}(p(y),1)-\A_{\HH_g}(y)\\
    &=\A_{g'(\HH(y))\HH}(y)-\A_{\HH_g}(y)\\
    &=\HH_g(y)-g'(\HH(y))\HH(y)\\
    &=g(\HH(y))-g'(\HH(y))\HH(y)\\
    &\leq\sup_{e\in E}\{g(\HH(e))-g'(\HH(e))\HH(e)\}\\
    &=\sup\{g(\HH(e))-g'(\HH(e))\HH(e)\mid \HH(e)\leq\max H+2\delta\}\\
    &\leq\sup\{g(\HH(e))\mid \HH(e)\leq\max H+2\delta\}\\
    &\qquad-\inf\{g'(\HH(e))\HH(e)\mid \HH(e)\leq\max H+2\delta\}\\
    &=\max H+2\delta-\frac{1}{\mathfrak{m}}\min H\\
    &\leq\max H
\eea
The last inequality holds by definition of $\delta$:
\beq
2\delta<\min\left\{\frac{\min(H)}{\mathfrak{m}},\frac{||H||}{\mathfrak{m}-1}\right\}
\eeq
Moreover, we used that $\min\HH=\min H$.

We recall the fact that $\A_H$ and $\A_\HH$ have the same critical points $\P_L^1(H)=\P_{L^N}^1(\HH)$ and critical values. In particular,
from the definition $a_{\max}(H)=\max\{\A_H(x)\mid x\in\P_L^1(H)\}$ it follows $\A_\HH(\xi)\leq a_{\max(H)}$. If we combine this with
the two previous inequalities we obtain

\beq
\A_H(x,\tau)\leq\A_{\HH_g}(y)+\max H\leq\A_{\HH}(\xi)+\max H\leq a_{\max}(H)+\max H
\eeq

The lower bound on $\A_H(x,\tau)$ is derived similarly by interchanging the roles of $\HH$ and $\HH_g$. This leads to:

\bean
\A_{\HH_g}(y)-\A_{\HH}(\xi)&=\int_{-\infty}^\infty\int_0^1\beta'(s)(\HH_g-\HH)(u)dtds+\int_{-\infty}^\infty\int_0^1|\partial_s u|^2dtds\\
&\geq-\sup(\HH_g-\HH)\\
&=-\sup\{\HH_g(e)-\HH(e)\mid\HH(e)\leq\max H+2\delta\}\\
&\geq\min H-\HH_g(\min H)
\eea

The last inequality follows from the fact that the function $g-\id:\R_{\geq0}\pf\R_{\geq0}$ is monotone decreasing and thus
the function $\HH_g(e)-\HH(e)=(g-\id)(\HH(e))$ is maximal at
$\min\HH=\min H$. From the inequality $\HH_g(\min H)\leq\max H+2\delta-\frac{1}{\mathfrak{m}}(||H||+2\delta)$ (see figure \ref{fig:graph_g}) we conclude

\bean
\A_{\HH_g}(y)-\A_{\HH}(\xi)&\geq\min H-\HH_g(\min H)\\
&\geq\min H -\big(\max H+2\delta-\frac{1}{\mathfrak{m}}(||H||+2\delta)\big)\\
&=-\big[||H||+2\delta-\frac{1}{\mathfrak{m}}(||H||+2\delta)\big]\\
&=\left(\frac{1}{\mathfrak{m}}-1\right)||H||-\left(1-\frac{1}{\mathfrak{m}}\right)2\delta\\
&\geq\left(\frac{1}{\mathfrak{m}}-1\right)||H||-\left(\frac{\mathfrak{m}-1}{\mathfrak{m}}\right)\frac{||H||}{\mathfrak{m}-1}\\
&=-||H||
\eea

In the second last inequality we used again the definition of $\delta$. Finally, we estimate

\bea
\A_H(x,\tau)-\A_{\HH_g}(y)&=\A_{\tau H}(\tau\ast(x,\tau))-\A_{\HH_g}(y)\\
    &=\A_{\tau H}(p(y),1)-\A_{\HH_g}(y)\\
    &=\A_{g'(\HH(y))\HH}(y)-\A_{\HH_g}(y)\\
    &=\HH_g(y)-g'(\HH(y))\HH(y)\\
    &=g(\HH(y))-g'(\HH(y))\HH(y)\\
    &\geq\inf\{g(\HH(e))-g'(\HH(e))\HH(e)\}\\
    &=0
\eea

and conclude

\beq
\A_H(x,\tau)\geq\A_{\HH_g}(y)\geq\A_\HH(\xi)-||H||\geq a_{\min}(H)-||H||\,.
\eeq

\end{proof}

\section{A counterexample}\label{section:counterex}

We consider a closed, symplectically aspherical, and integral symplectic manifold $(M,\om)$ which contains a Lagrangian sphere $L$ of dimension at least 2. As Paul Biran explained to us there are plenty of examples, see Example \ref{Ex:Biran_Ex}. According to Corollary \ref{cor:Bohr_Sommerfeld_existence_injective} there exists a Bohr-Sommerfeld pair $(E,\alpha)$ for $(M,\om,L)$ of power $N=1$.

On $L$ we choose a Morse function $f:L\pf\R$ with two critical points. We extend $f$ to a function $H:M\pf\R$. After a perturbation we can achieve that $H$ is non-resonant, see Theorem \ref{thm:non_resonant_is_generic}. Moreover, if we choose the perturbation small enough, we may assume that $H|_L$ still has exactly two critical points. After adding a suitable constant $H$ takes only positive values.

According to Proposition \ref{prop:tiny_Hamiltonian_and_Morse} there exists  ${\epsilon_0}>0$ such that for $0<\epsilon<\epsilon_0$ the Hamiltonian function ${\epsilon} H$ has exactly two Hamiltonian chords $x^\epsilon_\pm$ of Maslov index $\Mas(x^\epsilon_\pm)=\pm\frac{n}{2}$ and action value $-1<\A_{\epsilon H}(x^\epsilon_\pm)<0$. The action value estimate follows from equation \eqref{eqn:prop_tiny_Hamiltonian_and_Morse}. Since the power of the Bohr-Sommerfeld pair equals 1 Reeb chords of period $<1$ are 1-quantized chords of period $<1$, see Proposition \ref{pro:quantized_chords_are_Reeb_chords}.

We claim that $\epsilon_0H$ has no 1-quantized chords. Hamiltonian chords of $\epsilon_0H$ of period $\tau$ are Hamiltonian chords of $\tau\epsilon_0H$ of period $1$. Thus, for $0<\tau<1$ we have to compute the Hamiltonian chords of $\epsilon H$ for some $0<\epsilon<\epsilon_0$. From above we know that all of these have action values in the interval $(-1,0)$, thus none of them is 1-quantized. In particular, $\mathcal{R}_\L^1(H)=\emptyset$. This shows that in general the estimate \eqref{eqn:crucial_inequality} in Theorem A fails.

\begin{Ex}[Paul Biran]\label{Ex:Biran_Ex}
We take any projective algebraic manifold $M$ with $\pi_2(M)=0$.  Inside
$M$ we choose a sufficiently  high degree
hyperplane section $\Sigma$ such that there exists a Lefschetz pencil inside
$M$ whose generic fiber is symplectomorphic to $\Sigma$.
Since $\pi_2(M)=0$ the Lefschetz pencil necessarily has
singularities, see \cite[Section 5.1]{Biran_geometry_of_symplectic_intersections}. Thus, the vanishing cycles will give rise to Lagrangian
spheres in $\Sigma$. By the Lefschetz hyperplane theorem  $\pi_2(\Sigma)=0$  if  $ \dim_\R (\Sigma) \geq 6$. 
\end{Ex}

\begin{Rmk}\label{rmk:example_mu}
Choosing $\epsilon_0$ such that  $-\frac12<\A_{\epsilon H}(x^\epsilon_\pm)<0$ and $\min\epsilon H\leq\tfrac12$ for all $0<\epsilon\leq\epsilon_0$ the argument from above shows that the function $\mu$ introduced in Remark \ref{rmk:function_mu_nu} satisfies 
\beq
\mu(c)=0\quad\forall c\leq0\,.
\eeq
\end{Rmk}

\appendix
\section{Being non-resonant is a generic property}\label{appendix:non_resonant}

In this appendix we prove Theorem \ref{thm:non_resonant_is_generic} asserting that on a symplectic manifold $(M,\om)$ of dimension $\dim M\geq4$
a generic autonomous Hamiltonian function is non-resonant (see Definition \ref{def:non_resonant}). We first prove the following lemma.

\begin{Lemma}\label{lemma:generic_no_intersection}
If $\dim M\geq 4$ there exists an open and dense set $\mathscr{H}_1\subset C^\infty(M,\R)$ of smooth, autonomous Hamiltonian functions
such that for $H\in\mathscr{H}_1$ there are no solutions $(x,\sigma,\tau)\in C^\infty(\R,M)\times\R_{>0}\times\R_{>0}$
of the problem
\beq\label{eqn:Ham_Lagrn_loops}
\left\{
\begin{aligned}
\;\;&\dot{x}=X_H(x)\\
&x(t+\sigma)=x(t)\\
&x(0),\,x(\tau)\in L
\end{aligned}
\right.
\eeq
\end{Lemma}

\begin{proof}

There exists an open and dense set $\mathscr{H}_1\subset C^\infty(M,\R)$ of Hamiltonian functions $H$ for which
\begin{enumerate}
\item $\M_H:=\{(x,\sigma)\in C^\infty(\R,M)\times\R_{>0}\mid \dot{x}=X_H(x),\;x(t+\sigma)=x(t)\}$ is a two dimensional manifold and
\item $\ev:\M_H\times\R_{>0}\pf M\times M$ given by $\ev((x,\sigma),\tau)=(x(0),x(\tau))$ is transversal to $L\times L$.
\end{enumerate}
This implies that the space of solutions to problem \eqref{eqn:Ham_Lagrn_loops} is a smooth manifold of dimension
\beq
\dim \M_H+1-\mathrm{codim}(M\times M,L\times L)=3-2n<0
\eeq
since we assume $2n\geq4$, hence there exists no solution to \eqref{eqn:Ham_Lagrn_loops}.
\end{proof}
\begin{Lemma}\label{lemma:generic_Morse_on_L}
There exists an open and dense set $\mathscr{H}_2\subset C^\infty(M,\R)$ of Hamiltonian functions $H$ satisfying
\begin{itemize}
\item $H|_L:L\pf\R$ is Morse,
\item $\Crit H\cap L=\emptyset$.
\end{itemize}
\end{Lemma}
\begin{proof}
The restriction map $C^\infty(M,\R)\pf C^\infty(L,\R)$ is continuous and open. Thus, the pre-image of an open and dense subset of $C^\infty(L,\R)$
is open and dense in $C^\infty(M,\R)$. Since the Morse functions on $L$ form an open and dense subset of $C^\infty(L,\R)$ the first property
in the Lemma defines an open and dense subset of $C^\infty(M,\R)$.
The set of functions $f:M\pf\R$ with $\Crit(f)\cap L=\emptyset$ is open and dense in $C^\infty(M,\R)$. This implies the assertion.
\end{proof}
\begin{proof}[Proof of Theorem \ref{thm:non_resonant_is_generic}] It suffices to prove genericity of properties (1) -- (3) in Definition \ref{def:non_resonant} for a fixed $N\in\N$.\\

\noindent \underline{Step 1:}\quad Genericity of property (1) in definition \ref{def:non_resonant}.\\[1ex]
Since $C^\infty$ is not a Banach space we first work in the $C^k$ category and then deduce the $C^\infty$ case by a standard argument
due to Taubes, see \cite[Chapter 3.2]{McDuff_Salamon_J_holomorphic_curves_and_symplectic_topology}.
For $k\geq2$ we denote by $\mathscr{H}^k\subset C^k(M,\R)$ the open and dense space of Hamiltonian functions having no solutions to
problem \eqref{eqn:Ham_Lagrn_loops} and which satisfy the conditions of Lemma \ref{lemma:generic_Morse_on_L}. We claim that the space
\beq
\M:=\big\{(H,x)\in \mathscr{H}^k\times C^k(I,M)\mid \dot{x}=X_H(x),\;x(0),x(1)\in L\big\}
\eeq
is a Banach manifold. In order to prove this, we interpret $\M$ as zero-set of a section $s$ in a Banach space bundle $\E^k\pf\B^k$ as follows.
\begin{gather}
\B^k:=\mathscr{H}^k\times C^{k}\big((I,\partial I);(M,L)\big),\;\;\E^k_{(H,x)}:=\Gamma^{k-1}(x^*TM)\\[1ex]
s(H,x):=\dot{x}-X_H(x)
\end{gather}
We are required to prove that the vertical differential of $s$ along the zero-section is surjective. If $(H,x)\in\M=s^{-1}(0)$ then
this operator $D_{(H,x)}:T_{(H,x)}\B^k\pf\E^k_{(H,x)}$ is given by (after choosing a connection)
\beq
(\hat{H},\hat{x})\mapsto\nabla_t\hat{x}-\nabla_{\hat{x}}X_H(x)-X_{\hat{H}}(x)\,.
\eeq
To prove surjectivity of $D_{(H,x)}$ at $(H,x)\in\M=s^{-1}(0)$ we first show that the Hamiltonian chord $x$ is an injective map.
Otherwise, if there exist $t_0>t'_0$ such that $x(t_0)=x(t'_0)$ we conclude that $x(t)$ is $\tau=t_0-t'_0$ periodic, since $x$ solves the
autonomous ODE $\dot{x}=X_H(x)$. In particular, $(x,\tau)$ solves problem \eqref{eqn:Ham_Lagrn_loops}, unless $x$ is a constant map.
Since by assumption $H\in\mathscr{H}^k$ we are left with the case $x(t)=x_0\in L$ is constant.
Thus, the Hamiltonian function $H$ has a critical point at $x_0\in L$. This contradicts the second condition
in Lemma \ref{lemma:generic_Morse_on_L}.

Thus, the chord $x$ is injective. Therefore, for all $\eta\in\E^k_{(H,x)}$ there exists a function $\hat{H}$ defined in a neighborhood of $x$ such
that $X_{\hat{H}}(x(t))=\eta(t)$, hence $D_{(H,x)}(\hat{H},0)=\eta$ is surjective.

This shows that the space $\M$ is a Banach manifold. To prove that $\A_H$ is Morse for generic $H\in\mathscr{H}^k$ we consider the projection
$\pi=\mathrm{pr}_1:\M\pf\mathscr{H}^k$. We will show below that it is equivalent for $H$ to be a regular value of $\pi$ and for $\A_H$ to
be Morse. Thus, by the Sard-Smale Theorem, the action functional $\A_H$ is Morse for a generic Hamiltonian function $H\in\mathscr{H}^k$.

We now show the following equivalence: $H$ is a regular value of $\pi$ iff $\A_H$ is Morse.
For $(H,x)\in\M$ let $\pi(H,x):=H$ be a regular value of the projection, that is, $\forall \hat{H}\in\mathscr{H}^k\times C^{k}$ there exists
$\hat{x}\in\Gamma^{k-1}(x^*TM)$ such that $(\hat{H},\hat{x})\in T_{(H,x)}\M$. In particular,
\beq
\nabla_t\hat{x}-\nabla_{\hat{x}}X_H(x)-X_{\hat{H}}(x)=0
\eeq
Since $x$ is injective we can realize all vector fields in $\Gamma^{k-1}(x^*TM)$ in the form $X_{\hat{H}}(x)$ where
$\hat{H}$ ranges over all $C^k$-functions on $M$. In other words, $H$ is a regular value of $\pi$ if and only if the operator
$\hat{x}\mapsto\nabla_t\hat{x}-\nabla_{\hat{x}}X_H(x)$ is surjective. It is well-known that this operator is a Fredholm operator of Fredholm
index 0, thus,
it is surjective if and only if it is injective. We conclude that $H$ is a regular value of $\pi$ if and only if there is no non-constant solution
$\hat{x}$ to the equation $\nabla_t\hat{x}-\nabla_{\hat{x}}X_H(x)=0$, that is, if and only all critical points $x$ of the
action functional $\A_H$ is a nondegenerate.

Since being Morse is a $C^k$-open condition, the action functional $\A_H$ is Morse for a $C^k$-open and dense set of Hamiltonian functions.
We now deduce the $C^\infty$ assertion from the $C^k$ case.
Using that being Morse is an $C^k$-open and dense condition and that $C^\infty$ is dense in $C^k$ we can find for any $H\in C^\infty$ a
sequence $H_i^{(k)}\in C^\infty$ satisfying
\begin{itemize}
\item $H_i^{(k)}\stackrel{C^k}{\pf}H$ for $i\pf\infty$,
\item $\A_{H_i^{(k)}}$ is Morse.
\end{itemize}
Then the diagonal sequence $H_k^{(k)}$ converges in $C^\infty$ to $H$. Thus, the set of smooth Hamiltonian functions $H$ such that $\A_H$ is Morse
is dense in $C^\infty$. Moreover, being Morse is a $C^\infty$-open condition.

According to Lemma \ref{lemma:basic_lemma} and equation \eqref{eqn:basic_lemma_Lagr}
$\A_{\HH}$ is Morse if and only if $\A_H$ is Morse and the spectrum of $\A_H$ contains no value of the form $\frac{1}{2N}\Z$.
We will prove that this property holds for an open and dense set of Hamiltonian functions. We denote by $\mathscr{H}_3\subset C^\infty(M)$ the
open and dense subset
of Hamiltonian functions $H$ for which $\A_H$ is Morse. We consider the $\R$-action on $C^\infty(M)$ given by $H\mapsto H+r$ for $r\in\R$.
We observe that $\mathscr{H}_3$ is an $\R$-invariant subset. Since the spectrum of $\A_H$ for $H\in\mathscr{H}_3$ is a finite set it follows easily
that the set $\{H\in\mathscr{H}_3\mid \mathrm{Spec}\A_H\cap\frac{1}{2N}\Z=\emptyset\}$ is open and dense in $\mathscr{H}_3$ and hence also in
$C^\infty(M)$.\\[2ex]
\noindent \underline{Step 2:}\quad Genericity of property (2) in definition \ref{def:non_resonant}.\\

\noindent For $\tau>0$ we set $\mathscr{X}^k_{\tau}:=C^k([0,\tau],M)$. We fix $P\in\N$ and define
\beq
\mathscr{X}^k_P:=\bigcup_{\tau\in(0,P)}\{\tau\}\times\mathscr{X}^k_\tau
\eeq
is a (trivial) bundle over $(0,P)$ and set
\beq
\B^k_P:=\big\{(H,x,\tau)\in\mathscr{H}^k\times\mathscr{X}^k_P\mid x(0),x(\tau)\in L\big\}\;.
\eeq
The reason why we define the bundle $\mathscr{X}^k_P$ only over $(0,P)$ rather than over $(0,\infty)$ is that sequences of chords of bounded period $\tau$
converge according to Arzela-Ascoli. This will be used below in order to apply Taubes' procedure.
The tangent space of this Banach manifold $\B^k_P$ is given by
\beq
T_{(H,x,\tau)}\B^k_P=\left\{(\hat{H},\hat{x},\hat{\tau})\in C^k(M)\times\Gamma^k(x^*TM)\times\R\left|\;\;
        \begin{aligned}
        &\hat{x}(0)\in T_{x(0)}L\\
        &\hat{x}(\tau)+\hat{\tau}\dot{x}(\tau)\in T_{x(\tau)}L
        \end{aligned}
        \right.\right\}\;.
\eeq
We define a Banach bundle $\E^k\pf\B^k_P$ with fibers
\beq
\E^k_{(H,x,\tau)}:=\Gamma^{k-1}(x^*TM)\times\R\;.
\eeq
For $m\in\frac{1}{2N}\Z$ the zero-set of the section $s_m:\B^k_P\pf\E^k$ defined by
\beq
s_m(H,x,\tau):=(\dot{x}-X_H(x),\A_H(x,\tau)-m)
\eeq
equals
\beq
\M(m,P):=\big\{(H,x,\tau)\in\B^k_P\mid\dot{x}=X_H(x),\;\A_H(x,\tau)=m\big\}\;.
\eeq
In order to show that $\M(m,P)$ is a Banach manifold we show that the operator
\bea
D_{(H,x,\tau)}:T_{(H,x,\tau)}\B^k_P&\pf\E^k_{(H,x,\tau)}\\
(\hat{H},\hat{x},\hat{\tau})&\mapsto\Big(\nabla_t\hat{x}-\nabla_{\hat{x}}X_H(x)-X_{\hat{H}}(x),-\int_0^\tau\hat{H}(x)dt-H(x)\hat{\tau}\Big)
\eea
is surjective along the zero-section.
Given $(\eta,r)\in\E^k_{(H,x,\tau)}=\Gamma^{k-1}(x^*TM)\times\R$ we proved in Step (1) that there exists $(\hat{H},\hat{x})$ such that
\beq
\nabla_t\hat{x}-\nabla_{\hat{x}}X_H(x)-X_{\hat{H}}(x)=\eta\,.
\eeq
In fact, since $x$ is injective, we are free to choose $\hat{x}=0$. In light of the boundary condition
$\hat{x}(\tau)+\hat{\tau}\dot{x}(\tau)\in T_{x(\tau)}L$ this then forces $\hat{\tau}=0$. After setting
\beq
\widetilde{H}:=\hat{H}-\frac{1}{\tau}\left(r+\int_0^\tau\hat{H}(x)dt\right)
\eeq
it follows
\beq\label{eqn:Ds_surjective}
D_{(H,x,\tau)}(\widetilde{H},0,0)=(\eta,r)\,,
\eeq
that is, $D_{(H,x,\tau)}$ is surjective along the zero-section. We define
\bea
\phi:\B^k_P&\pf TM\times TM\\
(H,x,\tau)&\mapsto (\dot{x}(0),\dot{x}(\tau))\,.
\eea
To compute $d\phi$ we recall that there exists a canonical involution $\iota:TTM\pf TTM$ defined as follows. We think of an element in
$TTM$ as an equivalence class of maps $v:(-\epsilon,\epsilon)\times(-\epsilon,\epsilon)\pf M$. Then on representatives the involution $\iota$
is defined by $\iota(v)(s,t):=v(t,s)$. In particular, $v\in T_zTM$ is mapped to $\iota(v)\in T_{d\pi(z)v}TM$ where $\pi:TM\pf M$ is the projection.
We compute
\bea
d\phi(H,x,\tau):T_{(H,x,\tau)}\B^k_P&\pf T_{(\dot{x}(0),\dot{x}(\tau))}(TM\times TM)\\
(\hat{H},\hat{x},\hat{\tau})&\mapsto \Big(\iota\big(\dot{\hat{x}}(0)\big),\iota\big(\dot{\hat{x}}(\tau)\big)+\hat{\tau}\ddot{x}(\tau)\Big)\,.
\eea
In order to apply Lemma \ref{lemma:Dietmar} (see below) we need to check that $Ds_m(H,x,\tau)|_{\ker d\phi(H,x,\tau)}$ is surjective and that
$d\phi(H,x,\tau)$ is surjective.
The latter is obvious. The former follows from the above computation leading to equation \eqref{eqn:Ds_surjective}. Indeed,
$Ds_m(H,x,\tau)|_{C^{\infty}(M)\times\{0\}\times\{0\}}$ already is surjective and $C^{\infty}(M)\times\{0\}\times\{0\}\subset\ker d\phi(H,x,\tau)$.
Lemma \ref{lemma:Dietmar} implies that $\phi|_{\M(m,P)}:\M(m,P)\pf TM\times TM$ is a submersion.

We fix an auxiliary Riemannian metric on $M$ and consider the submanifold $T^\bot L\subset TM$ of all vectors perpendicular to $TL$. Then, since
$\phi|_{\M(m,P)}:\M(m,P)\pf TM\times TM$ is a submersion, the moduli space
\beq
\M^\bot(m,P):=\M(m,P)\cap\phi^{-1}(T^\bot L\times T^\bot L)
\eeq
is a smooth manifold. Since the period $\tau$ in $(H,x,\tau)$ is bounded, the set of regular Hamiltonian functions, that is, the regular values
of the projection $\pi:\M^\bot(m,P)\pf\mathscr{H}^k$, is open and dense.

As in Step (1) the Sard-Smale Theorem and the procedure of Taubes gives rise to a generic set $\mathscr{H}(m,P)$ of
smooth Hamiltonian functions. Then each Hamiltonian function in the generic set $\displaystyle\bigcap_{m,P}\mathscr{H}(m,P)$
satisfies the requirement (2) in definition \ref{def:non_resonant}.
\end{proof}

We learned the following Lemma from Dietmar Salamon.

\begin{Lemma}\label{lemma:Dietmar}
Let $\E\pf\B$ be a Banach bundle and $s:\B\pf\E$ a smooth section. Moreover, let $\phi:\B\pf N$ be a smooth map into the Banach manifold $N$.
We fix a point $x\in s^{-1}(0)\subset\B$ and set $K:=\ker d\phi(x)\subset T_x\B$ and assume the following two conditions.
\begin{enumerate}
\item The vertical differential $Ds|_K:K\pf\E_x$ is surjective.
\item $d\phi(x):T_x\B\pf T_{\phi(x)}N$ is surjective.
\end{enumerate}
Then  $d\phi(x)|_{\ker Ds(x)}:\ker Ds(x)\pf T_{\phi(x)}N$ is surjective.
\end{Lemma}

\begin{proof}
We fix $\xi\in T_{\phi(x)}N$. Condition (2) implies that there exists $\eta\in T_x\B$ satisfying $d\phi(x)\eta=\xi$. Condition (1)
implies that there exists $\zeta\in K\subset T_x\B$ satisfying $Ds(x)\zeta=Ds(x)\eta$. We set $\tau:=\eta-\zeta$ and compute
\beq
Ds(x)\tau=Ds(x)\eta-Ds(x)\zeta=0
\eeq
thus, $\tau\in\ker Ds(x)$. Moreover,
\beq
d\phi(x)\tau=d\phi(x)\eta-\underbrace{d\phi(x)\zeta}_{=0}=d\phi(x)\eta=\xi
\eeq
proving the Lemma.
\end{proof}

\section{Autonomous Hamiltonian systems with Lagrangian boundary conditions}\label{appendix:autonomous_Lagranians}

The main result of this appendix is Lemma \ref{lemma:finitely_many_quantized_chords_APPENDIX} stating that under certain assumptions
the number of $N$-quantized chords is finite. We close this section with two examples demonstrating that these assumptions are necessary.

Throughout this section $(M,\om)$ is a symplectic manifold and $L\subset M$ is a Lagrangian submanifold.

\begin{Prop}\label{prop:HZ_Poincare_family}
Let $H:M\pf\R$ be an autonomous Hamiltonian function. We assume that there exists a point $x\in L$ and $\tau>0$ such that
$x_\tau:=\phi_H^{\tau}(x)\in L$, $D\phi_H^{\tau}(T_{x}L)\pitchfork T_{x_\tau}L$, and  $X_H(x(0))\not\in T_{x(0)}L$
and $X_H(x(\tau))\not\in T_{x(\tau)}L$.
Then there exist unique (up to reparametrization), smooth families $s\mapsto x_H(s)\in L$
and $s\mapsto \tau_H(s)$ for $s\in(-\epsilon,\epsilon)$ such that $x_H(0)=x$, and $\tau_H(0)=\tau$ and
\beq
\phi_H^{\tau_H(s)}(x_H(s))\in L\,,\quad \tau_H'(0)\neq0\quad\text{and}\quad x'_H(0)\neq0\,.     
\eeq
\end{Prop}
\begin{Def}
In the situation of the above Proposition we denote the induced Hamiltonian chords by
\beq
\hat{x}_H(s,t):=\phi_H^t(x_H(s)) \quad\forall t\in[0,\tau_H(s)]\,.
\eeq
\end{Def}

\begin{Rmk}
The corresponding statement of the above proposition in the periodic case was known to \Poincare and is proved Chapter 4.1 of
the book \cite{Hofer_Zehner_Book}. More precisely, in Proposition 2 in Chapter 4.1 of \cite{Hofer_Zehner_Book} it is proved that
the above family $x_H(s)$ can be chosen to be parameterized by energy, that is $H(x_H(s))=H(x)+s$.

We point out that this stronger assertion does not hold in the relative case, in general, as Example \ref{ex:counterexample_HZ} shows.
\end{Rmk}

To prove Proposition \ref{prop:HZ_Poincare_family} we need the following
\begin{Lemma}
If $X_H(x)\not\in T_xL$ there exists $\xi\in T_xL$ with the property
\beq
dH(x)\xi\neq0\,.
\eeq
In particular, $T_xL\pitchfork T_x\Sigma$, where $\Sigma=H^{-1}(H(x))$ is the level set through $x$.
\end{Lemma}

\begin{proof}
We assume by contradiction that
\beq
0=dH(x)\xi=\om(X_H(x),\xi)\qquad\forall \xi\in T_xL\,.
\eeq
This implies that $X_H(x)\in\big(T_xL\big)^\om=T_xL$. This contradiction proves the Lemma.
\end{proof}

\begin{proof}[Proof of Proposition \ref{prop:HZ_Poincare_family}]
Differentiating the equation
\beq
H(\phi_H^t(x))=H(x)
\eeq
yields
\beq
dH(\phi_H^t(x))D\phi_H^t(x)=dH(x)\,.
\eeq
Thus, since $dH(x)\neq0$, we can choose a small neighborhood $U$ of $x$ and $\epsilon>0$ such that on the open set
$V:=\{\phi_H^t(x)\mid x\in U,\,t\in(-\epsilon,\tau+\epsilon)\}$ the function $H|_V$ has only regular values.

To prove the proposition we follow closely the proof of Proposition 2 in Chapter 4.1 of \cite{Hofer_Zehner_Book}.
Due to the assumption $X_H(x(0))\not\in T_{x(0)}L$
and $X_H(x(\tau))\not\in T_{x(\tau)}L$ we can choose two local hypersurface  $\Sigma_i\subset M$, $i=0,1$ in a neighborhood $U_0$ of $x$ and
$U_1$ of $x_\tau$ with the property
\bea
T_x\Sigma_0\oplus<X_H(x)>=T_xM\quad&\text{and}\quad L\cap U_0\subset\Sigma_0\\
T_{x_\tau}\Sigma_1\oplus<X_H(x_\tau)>=T_{x_\tau}M \quad&\text{and}\quad L\cap U_1\subset\Sigma_1\,.
\eea
Moreover, (for sufficiently small neighborhoods $U_i$) there exists a smooth function $\tau$ with $\tau(x)=\tau$ such that
\beq
\psi(y)=\phi^{\tau(y)}_H(y):\Sigma_0\pf\Sigma_1
\eeq
is well-defined. As in the proof of Lemma 1 in Chapter 4.1 of \cite{Hofer_Zehner_Book} it follows that
\beq
D\phi_H^\tau(x)=\begin{pmatrix}
  d\psi(x)&0\\
  \star&1
\end{pmatrix}\,.
\eeq
$D\phi_H^\tau(x)$ is nondegenerate since it is a symplectic transformation. Thus, $d\psi(x)$ is nondegenerate. We choose local coordinates on $\Sigma_i$
such that the Lagrangian submanifold $L$ corresponds to $\R^n\oplus\{0\}\subset\R^{2n-1}$ in both coordinate systems. We denote the map $\psi$
in local coordinates by
\beq
\widetilde{\psi}:\R^{2n-1}\pf\R^{2n-1}
\eeq
and assume that $\widetilde{\psi}(0)=0$. With respect to the splitting $\R^{2n-1}=\R^n\oplus\R^{n-1}$ we write
\beq
\widetilde{\psi}(x_1,x_2)=\big(\widetilde{\psi}_1(x_1,x_2),\widetilde{\psi}_2(x_1,x_2)\big)\,,
\eeq
and abbreviate
\beq
d\widetilde{\psi}(0)=
\begin{pmatrix}
\partial_{x_1}\widetilde{\psi}_1&\partial_{x_2}\widetilde{\psi}_1\\
\partial_{x_1}\widetilde{\psi}_2&\partial_{x_2}\widetilde{\psi}_2
\end{pmatrix}=:
\begin{pmatrix}A&B\\C&D
\end{pmatrix}\,.
\eeq
We claim that $\partial_{x_1}\widetilde{\psi}_2$ has full rank. Indeed, from the transversality $D\phi_H^{\tau}(T_{x}L)\pitchfork T_{x_\tau}L$
it follows (in local coordinates) that
\beq
d\widetilde{\phi}_H^\tau(0)\cdot
\begin{pmatrix}
a\\0\\0
\end{pmatrix}=
\begin{pmatrix}A&B&0\\C&D&0\\F_1&F_2&1
\end{pmatrix}\cdot
\begin{pmatrix}
a\\0\\0
\end{pmatrix}=
\begin{pmatrix}
Aa\\Ca\\F_1a
\end{pmatrix}\neq
\begin{pmatrix}
\star\\0\\0
\end{pmatrix}
\eeq
for all $a\not=0\in\R^n$. Since $F_1$ is a $1\times n$-matrix the above inequality readily implies that $\dim\ker C=1$. Hence,
$C=\partial_{x_1}\widetilde{\psi}_2$ has full rank. This implies that locally $\widetilde{\psi}_2^{-1}(0)$ is a 1-dimensional submanifold
of $L$.

We choose $x_H(s)$ to be a parametrization of the local 1-manifold $\widetilde{\psi}_2^{-1}(0)$. This includes that assertion
$x_H'(0)\neq0$. $\tau_H(s)$ is defined accordingly. It remains to be proved that $\tau_H'(0)\neq0$.

Let us assume by contradiction that $\tau_H'(0)=0$. We recall the notation $\tau=\tau_H(0)$, $x=x_H(0)$ and
$x_\tau=\phi_H^{\tau}(x)=\phi_H^{\tau_H(0)}(x(0))$. Then the following holds
\bea
D\phi_H^{\tau}\big(T_{x}L\big)\ni D\phi_H^{\tau_H(0)}(x(0))\cdot x_H'(0)&=\frac{\partial}{\partial s}\Big|_{s=0}\phi_H^{\tau_H(0)}(x_H(s))\\[1ex]
             &=\frac{\partial}{\partial s}\Big|_{s=0}\underbrace{\phi_H^{\tau_H(s)}(x_H(s))}_{\in L}\in T_{x_\tau}L
\eea
where we used $\tau_H'(0)=0$ in the second equation. The transversality assumption $D\phi_H^{\tau}(T_{x}L)\pitchfork T_{x_\tau}L$
implies that $x_H'(0)=0$. This contradiction concludes the proof.
\end{proof}

The following lemma is Lemma \ref{lemma:finitely_many_quantized_chords} on page \pageref{lemma:finitely_many_quantized_chords}.
\begin{Lemma}\label{lemma:finitely_many_quantized_chords_APPENDIX}
We assume that $L\subset (M,\om)$ is a closed, aspherical Lagrangian submanifold and that $H:M\pf (0,\infty)$ is a positive Hamiltonian function
satisfying the transversality conditions $D\varphi_H^\tau(T_{x(0)}L)\pitchfork T_{x(\tau)}L$, $X_H(x(0))\not\in T_{x(0)}L$,
and $X_H(x(\tau))\not\in T_{x(\tau)}L$ for all $N$-quantized chords.
Then the set $\P_L^\mathfrak{q}(H;\tau_0,N)$ of $N$-quantized chords with period less or equal than $\tau_0$ is finite.
\end{Lemma}

\begin{proof}
Let $(x,\tau)$ be a $N$-quantized chord, in particular, it is a critical point of the action functional $\A_{H,\tau}:C^\infty\big([0,\tau];M,L\big)\pf\R$
given by
\beq
\A_{H,\tau}(x):=\A_H(x,\tau)
\eeq
where $\A_H(x,\tau)$ is defined in Definition \ref{def:Ham_chord_and_general_action_functional}, that is
\beq
d\A_{H,\tau}(x)=0\,.
\eeq

Let us assume that there exists a sequence $(x_\nu,\tau_\nu)\in\P_L^\mathfrak{q}(H;\tau_0,N)$. Since $M$ and $L$ are compact and $(\tau_\nu)$ is bounded
the Arzela-Ascoli theorem implies that a subsequence $(x_\nu,\tau_\nu)$ converges to an element $(x,\tau)\in\P_L^\mathfrak{q}(H;\tau_0,N)$. For
$\nu$ large enough the subsequence $(x_\nu(0),\tau_\nu)$ is part of a local family $(x_H(s),\tau_H(s))$ given by
Proposition \ref{prop:HZ_Poincare_family}. We assume by contradiction that the convergent subsequence is non-constant.
Because all $(x_\nu,\tau_\nu)$ are $N$-quantized we have
\beq
\frac{\partial}{\partial s}\Big|_{s=0}\A_H(\hat{x}_H(s),\tau_H(s))=0\,.
\eeq
On the other hand we compute using $\tau_H'(0)\neq0$, $H>0$, $\hat{x}_H(0)=x$ and $\tau_H(0)=\tau$
\bea
\frac{\partial}{\partial s}\Big|_{s=0}\A_H(\hat{x}_H(s),\tau_H(s))
            &=\frac{\partial}{\partial s}\Big|_{s=0}\A_{H,\tau}(\hat{x}_H(s))-\frac{\partial}{\partial s}\Big|_{s=0}(\tau_H(s)-\tau)H(x_H(s))\\[1ex]
            &=\underbrace{d\A_{H,\tau}(x)}_{=0}\cdot \hat{x}_H'(0)-\tau_H'(0)H(x)\\[1ex]
            &=-\tau_H'(0)H(x)\neq0
\eea
This contradiction concludes the proof.
\end{proof}

We conclude this section with two examples showing that the condition that the Hamiltonian function is positive is necessary. Moreover,
they show that the family of Hamiltonian chords from Proposition \ref{prop:HZ_Poincare_family} cannot be parameterized by energy as opposed
to the periodic case.

\begin{Ex}\label{ex:counterexample}
In figure \ref{fig:counterexample} we assume that the area of the grey-shaded region equals an integer. Then there are uncountably many quantized
chords connecting the point $P$ and $Q_s$ inside $\{H=0\}$ where the point $Q_s$ locally varies on $\{H=0\}$.

\begin{figure}[htb]
\input{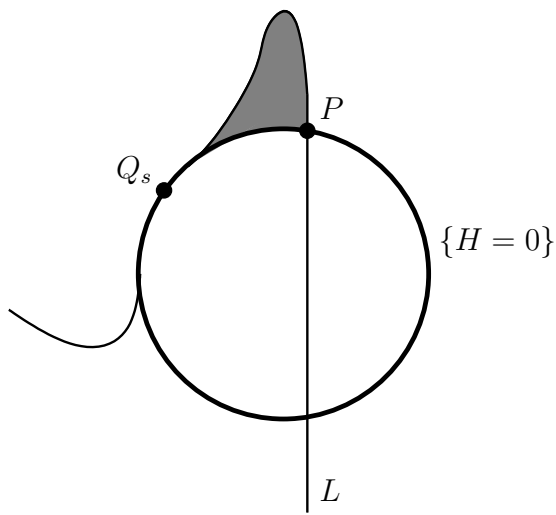}
\caption{Infinitely many quantized chords}\label{fig:counterexample}
\end{figure}
\end{Ex}

\begin{Ex}\label{ex:counterexample_HZ}
We construct an example of two Lagrangian submanifolds and a Hamiltonian function such that the Hamiltonian vector field intersects both Lagrangian
submanifolds transversely. Moreover, the Hamiltonian flow has a one-parametric family of Hamiltonian chords of constant energy. In particular,
this family cannot be parameterized by energy.

In $\R^4$ with coordinates $(x_1,x_2,y_1,y_2)$ and symplectic form $\sum dx_i\wedge dy_i$ we consider the following two Lagrangian submanifolds
\bea
L_1&:=\{x_1=x_2=0\}\\
L_2&:=e_1+<X,Y>
\eea
where $e_1:=(1,0,0,0)$, $X:=(a,0,0,b)$, and $Y:=(0,a,b,0)$, for $a,b\neq0$ to be determined later. We note that
\beq
\om(X,Y)=ab-ab=0\,.
\eeq
We set
\beq
H(x_1,x_2,y_1,y_2):=y_1\,\qquad\text{thus}\qquad X_H=\frac{\partial}{\partial x_1}\,.
\eeq
In particular,
\beq
\phi_H^\tau(x_1,x_2,y_1,y_2)=(x_1,x_2,y_1,y_2)+\tau e_1\,.
\eeq
Obviously, $X_H$ intersects $L_0$, $L_1$ transversely. Moreover, $\phi_H^1(L_1)=(e_1+L_1)\pitchfork L_2$.
According to Proposition \ref{prop:HZ_Poincare_family} (which obviously holds also for two transverse Lagrangian submanifolds)
there exists locally a one-parametric family $x_H(s)\in L_1$ and $\tau_H(s)$. In this example they are explicitly given by
\beq
x_H(s):=(0,0,0,s)\in L_1\qquad \tau_H(s):=1+\frac{a}{b}s\,.
\eeq
Indeed,
\beq
\phi_H^{\tau(s)}(x_H(s))=(1+\frac{a}{b}s,0,0,s)=e_1+\frac{s}{b}X\in L_2\,.
\eeq
We observe that for $a\neq0$ the period $\tau_H(s)$ is non-constant while
\beq
H(x_H(s))=0\,.
\eeq
We note that the intersection point of $L_1$ and $L_2$ is given by
\beq
L_1\cap L_2=\{(0,0,0,-\frac{b}{a})\}\,.
\eeq
The symplectic area of the Hamiltonian chord $\phi_H^{\tau(s)}(x_H(s))$ relative to $L_1$ and $L_2$ obviously vanishes, since the affine subspace
containing the intersection point $(0,0,0,-\frac{b}{a})$ and the chord is Lagrangian. There exists a representative $L_1\#L_2$
of the Lagrangian isotopy class of the Lagrangian connected sum for which $\phi_H^{\tau(s)}(x_H(s))$ is still a Hamiltonian chord with vanishing
symplectic area. In particular, $\phi_H^{\tau(s)}(x_H(s))$ are quantized chords for all $s$.
\end{Ex}

\section{Transversal intersection for quantized chords}\label{appendix:quantized_chords}

The main result in this appendix is Proposition \ref{prop:implying_Morse_for_A_g(H)} which is crucial for establishing the fact that
the action functional $\A_{g(\HH)}$ is Morse, see Lemma \ref{lemma:action_fctl_A_g(HH)_is_Morse}.

We use the notation introduced in Section \ref{sec:the_setting}. Here are the essentials:
$p:E\pf M$ is a complex line bundle. $\HH$ is the fiber-wise quadratic lift to $E$ of a Hamiltonian function on the base $M$. The flow of a
Hamiltonian function $H$ is denoted by $\phi_H^\tau$, and the Hamiltonian vector field by $X_H$. The function $\HH_g=g(\HH)$. $L\subset(M,\om)$
is a Lagrangian submanifold.

\begin{Lemma}\label{lemma:flow_commutes_with_projection}
The following two equations hold
\beq
p\circ\phi_\HH^\tau=\phi_H^\tau\circ p
\eeq
and
\beq
dp(\phi_\HH^\tau(x))\circ D\phi_\HH^\tau(x)=D\phi_H^\tau(p(x))\circ dp(x)\,.
\eeq
Moreover, for $x\in E$ and $\xi\in T_xE$
\beq
\phi_{\HH_g}^1(x)=\phi_\HH^{g'(\HH)}(x)
\eeq
and
\beq
D\phi_{\HH_g}^1(x)\cdot\xi=D\phi_\HH^{g'(\HH)}(x)\cdot\xi+g''(\HH(x))\,\big(d\HH(x)\cdot\xi\big)\,X_\HH(\phi_{\HH_g}^1(x))
\eeq
\end{Lemma}
\begin{proof}
Integrating the equations \eqref{eqn:Ham_vfield_of_HH} and \eqref{eqn:Ham_vfield_of_HH_vertical} with respect to $\tau$ leads to
the first equation. Differentiating with respect to $x$ gives the second. The third equation follows from the transformation rule
\beq
X_{\HH_g}=g'(\HH)X_\HH\,.
\eeq
We recall that $g'(\HH)$ is constant along chords of $\HH_g$, see Remark \ref{rmk:basic_rmk}. The last again by differentiating.
\end{proof}

\begin{Lemma}\label{lemma:degeneracy_for_Hamiltonian}
Assume that there exists $x\in L^N\setminus L$ and $\tau\in\R$, such that $\phi_\HH^\tau(x)\in L^N$
and $D\phi^\tau_H(p(x))\big(T_{p(x)}L\big)\pitchfork T_{\phi^\tau_H(p(x))}L$ holds. Then
\beq
D\phi_\HH^\tau(T_xL^N)\cap T_{\phi_{\HH}^\tau(x)}L^N=T^v_{\phi_{\HH}^\tau(x)}L^N=<X>\,,
\eeq
where $X$ is the Liouville vector field, see equation \eqref{eqn:Liouville_vfield}.
\end{Lemma}

\begin{proof}
This follows immediately from Lemma \ref{lemma:lifted_flow_preserves_horizontal_distribution}.
\end{proof}

\begin{Prop}\label{prop:implying_Morse_for_A_g(H)}
Let $g:\R\pf\R$ be a smooth function and recall the notation $\HH_g=g(\HH)$.
Let $x\in L^N\setminus L$ such that $\phi_{\HH_g}^1(x)\in L^N$ and $H(p(x))\neq0$. We assume
$D\phi^\tau_H(p(x))\big(T_{p(x)}L\big)\pitchfork T_{\phi^\tau_H(p(x))}L$, where $\tau:=g'(\HH(x))$.
If $g''(\HH(x))\neq0$ then
\beq
D\phi_{\HH_g}^1(x)\big(T_xL^N)\pitchfork T_{\phi_{\HH_g}^1(x)}L^N
\eeq
holds.
\end{Prop}

\begin{proof}
We pick $\eta\in T_xL^N$ and assume by contradiction that $D\phi_{\HH_g}^1(x)\cdot\eta\in T_{\phi_{\HH_g}^1(x)}L^N$.

\noindent \underline{Step 1}:\quad We show that $d\HH(x)\cdot\eta=0$.\\[1ex]

We write $\eta=\eta^h+c X(x)$, where $\eta^h$ is horizontal and $c\in\R$. Lemma \ref{lemma:flow_commutes_with_projection} asserts
\beq\label{eqn:differential_of_phi_HH_g}
D\phi_{\HH_g}^1(x)\cdot\eta=D\phi_\HH^{\tau}(x)\cdot\eta+g''(\HH(x))\,\big(d\HH(x)\cdot\eta\big)\,X_\HH(\phi_{\HH_g}^1(x))\,.
\eeq
Since we assume $D\phi_{\HH_g}^1(x)\cdot\eta\in T_{\phi_{\HH_g}^1(x)}L^N$ we can again write
\beq
D\phi_{\HH_g}^1(x)\cdot\eta=\zeta^h+b X(\phi_{\HH_g}^1(x))\,.
\eeq
We compute
\bea
0=\alpha\big(\zeta^h+b X(\phi_{\HH_g}^1(x))\big)&=\alpha\big(D\phi_{\HH_g}^1(x)\cdot\eta\big)\\
&=\alpha\big(D\phi_\HH^{\tau}(x)\cdot\eta\big)+\alpha\big(g''(\HH(x))\,\big(d\HH(x)\cdot\eta\big)\,X_\HH(\phi_{\HH_g}^1(x))\big)\\
&=\alpha\big(\eta\big)+g''(\HH(x))\,\big(d\HH(x)\cdot\eta\big)\alpha\big(X_\HH(\phi_{\HH_g}^1(x))\big)\\
&=\alpha\big(\eta^h+c X(x)\big)+g''(\HH(x))\,\big(d\HH(x)\cdot\eta\big)[-NH(p(x))]\\
&=-N\cdot \underbrace{H(p(x))g''(\HH(x))}_{\neq0}\,\big(d\HH(x)\cdot\eta\big)
\eea
where we used that $\alpha$ vanishes on horizontal vectors and the Liouville vector field $X$, moreover, that $\alpha$ is preserved by $\phi_\HH^\tau$,
see Lemma \ref{lemma:lifted_flow_preserves_horizontal_distribution}, and the explicit form of $X_\HH$, see
equations \eqref{eqn:Ham_vfield_of_HH} and \eqref{eqn:Ham_vfield_of_HH_vertical}.
We conclude that $d\HH(x)\cdot\eta=0$.\\

\noindent \underline{Step 2}:\quad We prove that $\eta^h=0$.\\[1ex]

The assumption $D\phi_{\HH_g}^1(x)\cdot\eta\in T_{\phi_{\HH_g}^1(x)}L^N$ together with $d\HH(x)\cdot\eta=0$
and equation \eqref{eqn:differential_of_phi_HH_g} implies
\beq
D\phi_{\HH_g}^1(x)\cdot\eta=D\phi_\HH^{\tau}(x)\cdot\eta\in T_{\phi_{\HH}^\tau(x)}L^N\,.
\eeq
Since $\eta\in T_xL^N$
\beq
D\phi_\HH^{\tau}(x)\cdot\eta\in D\phi_\HH^{\tau}(x)(T_xL^N)\cap T_{\phi_{\HH}^\tau(x)}L^N
\eeq
therefore, Lemma \ref{lemma:degeneracy_for_Hamiltonian} implies that
\beq
D\phi_\HH^{\tau}(x)\cdot\eta\in \,<X(\phi_{\HH}^\tau(x))>\,= T_{\phi_{\HH}^\tau(x)}^vL^N\,.
\eeq
Since $\phi_\HH^\tau$ preserves the Liouville vector field $X$ we conclude
from  $\eta=\eta^h+c X(x)$ that $\eta^h=0$.\\

\noindent \underline{Step 3}:\quad We prove that $\eta=0$.\\[1ex]

From Steps 1 and 2 we conclude $d\HH(x)\cdot\eta$ and $\eta=cX$, thus we compute
\bea
0=d\HH(cX)\cdot\eta&=\big(Nf'(r)H(p(x))dr+Nf(r)dH(x)\big)\cdot cX\\
    &=Nf'(r)H(p(x))dr\cdot cX\\
    &=c Nf(r)H(p(x))\,.
\eea
In particular, we obtain from $H(p(x))\neq0$ and $f(r)\neq0$, that $c=0$ and therefore $\eta=0$.
\end{proof}

\section{Holonomy of line bundles}\label{appendix:holonomy}

Let $\pi:\E\pf M$ be a principle $S^1$-bundle with connection 1-form $\alpha$. We recall the following explicit formula for the holonomy around a loop $\gamma:S^1\pf M$ in terms of a connection 1-form $\alpha$
\beq\label{eqn:def_of_holonomy}
\mathrm{hol}_\alpha(\gamma)=-\int_0^1 \eta^* \alpha\in S^1=\R/\Z\;.
\eeq
Here $\eta:S^1\pf\E$ is a loop satisfying $\pi\circ\eta=\gamma$. Alternatively, the holonomy $\mathrm{hol}_\alpha(\gamma)\in S^1$ is determined by
\beq
P_\gamma^\alpha(e)= \mathrm{hol}_\alpha(\gamma).e
\eeq
where, $e\in\E_{\gamma(0)}$, $P_\gamma^\alpha:\E_{\gamma(0)}\pf\E_{\gamma(0)}$ denotes the parallel transport along $\gamma$ with respect to the connection $\alpha$, and $g.e$ denotes the $S^1$-action. More details can be found in the book \cite[Chapter II]{Kobayashi_Nomizu_Vol_I}.

\begin{Prop}\label{prop:holonomy_appendix}
Let $(\E,\alpha)$ and $(\F,\beta)$ be principal $S^1$-bundles with connection 1-forms over the manifold $M=\E/S^1=\F/S^1$. Then the following holds.
\begin{enumerate}
\item There exists a canonical connection 1-form $\alpha\otimes\beta$ on the $S^1$-bundle $\E\otimes\F$. Moreover, the holonomy
$\mathrm{hol}_\alpha:C^\infty(S^1,M)\pf S^1$ satisfies
\beq
\mathrm{hol}_{\alpha\otimes\beta}=\mathrm{hol}_\alpha+\mathrm{hol}_\beta\,.
\eeq
\item There exists a canonical connection 1-form $\alpha^*$ on the dual $S^1$-bundle $\E^*$ and
\beq
\mathrm{hol}_{\alpha^*}=-\mathrm{hol}_\alpha\,.
\eeq
\item The bundle $(\E\otimes \E^*,\alpha\otimes\alpha^*)$ is canonically isomorphic to the trivial bundle $M\times S^1$ together with its
trivial connection.
\end{enumerate}
\end{Prop}

We only sketch the proof:\\

We think of a connection $\alpha$ in $\E$ as an $S^1$-invariant hyperplane distribution $H^\E$ which is transversal to the infinitesimal generator of the $S^1$-action. We construct $\E\otimes\F$. The fiber product $\E\times_M\F$ of $\E$ and $\F$ is defined as follows
\beq
\E\times_M\F=\{(e,f)\mid p_\E(e)=p_\F(f)\}\,.
\eeq
This is a principal $T^2$-bundle over $M$. We set $\overline{\Delta}:=\{(g,-g)\mid g\in S^1\}\subset T^2$ and define
\beq
\E\otimes\F:=(\E\times_M\F)/\overline{\Delta}
\eeq
which is a principal $T^2/\overline{\Delta}\cong S^1$-bundle. We denote by $H^\E$ resp.~$H^\F$ the hyperplane distributions on $\E$ resp.~$\F$. Then
\beq
H^{\E\times_M\F}:=dp_\E^{-1}(H^\E)\cap dp_\F^{-1}(H^\F)
\eeq
is a $T^2$-invariant codimension-2-distribution which is transversal to the infinitesimal generators of the torus action. In particular,
$H^{\E\times_M\F}$ descends to connection $H^{\E\otimes\F}$ on $\E\otimes\F$.

To compute the holonomy we recall that
for a loop $\gamma\in C^\infty(S^1,M)$ and $e\in \E_{\gamma(0)}$ the holonomy $\mathrm{hol}_\alpha(\gamma)\in S^1$ is determined by
\beq
P_\gamma^\alpha(e)= \mathrm{hol}_\alpha(\gamma).e
\eeq
where $P_\gamma^\alpha:\E_{\gamma(0)}\pf\E_{\gamma(0)}$ denotes the parallel transport along $\gamma$ with respect to the connection $\alpha$ and $g.e$ denotes the $S^1$-action. We observe on $\E\times_M\F$ that 
\beq
P_\gamma^{\alpha\times_M\beta}(e,f)=(P_\gamma^{\alpha}(e),P_\gamma^{\beta}(f))=(\mathrm{hol}_\alpha(\gamma).e, \mathrm{hol}_\beta(\gamma).f)\in(\E\times_M\F)_{\gamma(0)}\;.
\eeq
Thus, $\mathrm{hol}_{\alpha\otimes\beta}=\mathrm{hol}_\alpha+\mathrm{hol}_\beta$ holds.  Statement (2) about the holonomy is proved analogously.

We construct $\E^*$. We recall that $\E$ is a compact manifold with a free $S^1$-action $\psi:S^1\times\E\pf\E$. We define
$\psi^*:S^1\times\E\pf\E$ by $\psi^*(g,e):=\psi(-g,e)$. Then $\E^*$ is the principal $S^1$-bundle with total space $\E$ and action $\psi^*$.
Moreover, the connection $H^{\E^*}=H^\E$.
 
For (3) the canonical isomorphism is given by
\bea
\Phi:\E\otimes\E^*&\pf M\times S^1\\
[e,e^*]=[e,\psi^*(g,e)]&\mapsto (p_\E(e),g)\;.
\eea

$ $\\[1ex]

\noindent\hrulefill
%
%
\bibliographystyle{amsalpha}
\bibliography{../../../Bibtex/bibtex_paper_list}
\end{document}